\documentclass{amsart}
\usepackage{amssymb}
\usepackage[utf8]{inputenc}
\usepackage{setspace}
\usepackage{stmaryrd}
\usepackage{xcolor}
\usepackage{subfiles}
\usepackage{xifthen}
\usepackage{enumitem}

\definecolor{e-mail}{rgb}{0,.40,.80}
\definecolor{reference}{rgb}{.20,.60,.22}
\definecolor{citation}{rgb}{0,.40,.80}
\usepackage[colorlinks=true,
            linkcolor=reference,
            citecolor=citation,
            urlcolor=e-mail]{hyperref}

\usepackage[all]{xy}
\usepackage[capitalise]{cleveref}
\usepackage{eucal}

\usepackage{cite}

\usepackage{eso-pic}
\usepackage[disable]{todonotes}
\newcommand{\daj}[1]{\todo[caption={#1}]{ { \begin{spacing}{0.5}\scriptsize{#1-DJ}\end{spacing} } } }

\newcommand{\sltwo}{\mathfrak{sl}_2}
\newcommand{\cA}{\mathcal{A}}
\newcommand{\cC}{\mathcal{C}}
\newcommand{\C}{\mathbf{C}}
\newcommand{\cD}{\mathcal{D}}
\newcommand{\D}{\mathrm{D}}
\newcommand{\cE}{\mathcal{E}}
\newcommand{\cF}{\mathcal{F}}
\newcommand{\g}{\mathfrak{g}}
\newcommand{\h}{\mathfrak{h}}	
\renewcommand{\H}{\mathrm{H}}

\renewcommand{\k}{k}

\newcommand{\bL}{\mathbb{L}}

\newcommand{\N}{\mathrm{N}}
\newcommand{\cN}{\mathcal{N}}
\renewcommand{\O}{\mathcal{O}}
\newcommand{\cP}{\mathcal{P}}
\newcommand{\Q}{\mathbf{Q}}
\newcommand{\bR}{\mathbf{R}}
\newcommand{\R}{\mathrm{R}}

\newcommand{\T}{\mathrm{T}}

\newcommand{\Z}{\mathbf{Z}}
\newcommand{\rZ}{\mathrm{Z}}
\newcommand{\rt}[1]{\mathchoice{\underset{#1}{\otimes}}{\otimes_{#1}}{}{}}
\newcommand{\rtL}[1]{\mathchoice{\overset{\bL}{\underset{#1}{\otimes}}}{\otimes_{#1}^\bL}{}{}}

\newcommand{\Dq}[1][]{\mathrm{D}_q^{#1}}
\newcommand{\Oq}[1][]{\mathcal{O}_q^{#1}}
\newcommand{\U}{\mathrm{U}}
\newcommand{\Uq}{\mathrm{U}_q}

\newcommand{\bu}{\mathbf{1}}

\newcommand{\Ann}{\mathbb{A}\mathrm{nn}}
\newcommand{\act}{\mathrm{act}}
\newcommand{\bop}{\sigma\mathrm{op}}
\newcommand{\coev}{\mathrm{coev}}
\newcommand{\coinv}{\mathrm{coinv}}
\newcommand{\const}{\operatorname{const}}
\newcommand{\CP}{\mathbf{CP}}

\newcommand{\End}{\mathrm{End}}
\newcommand{\ev}{\mathrm{ev}}
\newcommand{\Fun}{\mathrm{Fun}}
\newcommand{\HC}{\operatorname{HC}}
\newcommand{\HH}{\mathrm{HH}}
\newcommand{\RHom}{\mathbb{R}\mathrm{Hom}}
\newcommand{\Hom}{\operatorname{Hom}}
\newcommand{\id}{\mathrm{Id}}
\newcommand{\inv}{\mathrm{inv}}
\newcommand{\Loc}{\mathrm{Loc}}

\newcommand{\op}{\mathrm{op}}
\newcommand{\Per}{\mathrm{Per}}
\newcommand{\PSL}{\mathrm{PSL}}
\newcommand{\QCoh}{\mathrm{QCoh}}
\newcommand{\SL}{\mathrm{SL}}

\newcommand{\str}{\mathrm{str}}
\newcommand{\Spec}{\mathrm{Spec}}
\newcommand{\TL}{\mathrm{TL}}
\newcommand{\triv}{\mathrm{triv}}
\newcommand{\ZMug}{\mathrm{Z}_{\mathrm{M\ddot{u}g}}}

\newcommand{\Bimod}{\mathrm{Bimod}}
\newcommand{\Cat}{\mathrm{Cat}}
\newcommand{\Disk}{\mathrm{Disk}}
\newcommand{\Mfld}{\mathrm{Mfld}}
\newcommand{\PrL}{{\mathrm{Pr}^{\mathrm{L}}}}
\newcommand{\Rep}{\mathrm{Rep}}
\newcommand{\Repfd}{\mathrm{Rep}^{\mathrm{fd}}}
\newcommand{\Repq}{\mathrm{Rep}_q}
\newcommand{\Repqfd}{\mathrm{Rep}^{\mathrm{fd}}_q}

\newcommand{\Sk}{\mathrm{Sk}}
\newcommand{\Skint}{\mathrm{Sk}^{\mathrm{int}}}
\newcommand{\SkAlg}{\mathrm{SkAlg}}
\newcommand{\SkAlgint}{\mathrm{SkAlg}^{\mathrm{int}}}
\newcommand{\SkCat}{\mathrm{SkCat}}
\newcommand{\SkMod}{\mathrm{SkMod}}
\newcommand{\Vect}{\mathrm{Vect}}
\newcommand{\Ext}{\mathrm{Ext}}

\newcommand{\fr}[1]{\widehat{#1}}
\newcommand{\rev}[1]{\breve{#1}}
\newcommand{\into}{\hookrightarrow}

\newcommand{\closure}[1]{\overline{#1}}

\newcommand{\DD}{\mathbb{D}}
\newcommand{\DDc}{\closure{\mathbb{D}}}
\newcommand{\DDout}{\mathbb{D}^{\mathrm{out}}}

\newcommand{\bimod}[3]{
  \ifthenelse{\isempty{#3}}
  {{}_{#1}\mathrm{BiMod}_{#2}}
  {{}_{#1}\mathrm{BiMod}_{#2}(#3)}
}
\newcommand{\lmod}[2]{
  \ifthenelse{\isempty{#2}}
  {\mathrm{LMod}_{#1}}
  {\mathrm{LMod}_{#1}(#2)}
}
\newcommand{\rmod}[2]{
  \ifthenelse{\isempty{#2}}
  {\mathrm{RMod}_{#1}}
  {\mathrm{RMod}_{#1}(#2)}
}

\newcommand{\llpar}{(\!(}
\newcommand{\rrpar}{)\!)}

\DeclareMathOperator{\colim}{colim}

\mathchardef\mhyphen="2D

\theoremstyle{plain}
\newtheorem{theorem}{Theorem}[section]
\newtheorem{proposition}[theorem]{Proposition}

\newtheorem{lemma}[theorem]{Lemma}
\newtheorem{corollary}[theorem]{Corollary}

\newtheorem{introthm}{Theorem}
\newtheorem{introcor}{Corollary}

\theoremstyle{definition}
\newtheorem{definition}[theorem]{Definition}
\newtheorem{example}[theorem]{Example}
\newtheorem*{introex}{Example}
\newtheorem{introrem}{Remark}

\theoremstyle{remark}
\newtheorem{remark}[theorem]{Remark}

\newcommand{\defterm}[1]{\textbf{\emph{#1}}}

\usepackage{tikz}
\newcommand{\tik}[1]{\begin{tikzpicture}[baseline=(current bounding box.center), scale=0.65] #1 \end{tikzpicture} }
\tikzstyle cross=[preaction={draw=white, -, line width=4pt}, thick]
\newcommand{\fieldgoal}[3][-]{\draw[cross,#1] (#2+2.5,-#3).. controls (#2+2.5,-#3-0.3) and (#2+1,-#3-0.7)..(#2+1,-#3-1); \draw[cross,#1] (#2,-#3).. controls (#2,-#3-0.3) and (#2+2.5,-#3-0.7)..(#2+2.5,-#3-1);\draw[cross,#1] (#2+1.5,-#3).. controls (#2+1.5,-#3-0.3) and (#2,-#3-0.7)..(#2,-#3-1);}

\usepgflibrary{shapes.geometric}
\usetikzlibrary{matrix}
\usetikzlibrary{shapes}

\usetikzlibrary{arrows.meta}

\newcommand{\KP}[1]{%
  \begin{tikzpicture}[baseline=-\dimexpr\fontdimen22\textfont2\relax, color=black, thick]
  #1
  \end{tikzpicture}%
}

\newcommand{\KPA}{%
  \KP{\filldraw[fill=none] circle (0.3);}%
}
\newcommand{\KPB}{%
  \KP{
    \draw (-0.3,-0.3) -- (0.3,0.3);
    \draw (-0.3,0.3) -- (-0.05,0.05);
    \draw (0.05,-0.05) -- (0.3,-0.3);
  }%
}
\newcommand{\KPC}{%
  \KP{%
    \draw (-0.3,-0.3) .. controls (0.05,0) .. (-0.3,0.3);
    \draw (0.3,-0.3) .. controls (-0.05,0) .. (0.3,0.3);
  }%
}
\newcommand{\KPD}{%
  \KP{%
    \draw (-0.3,0.3) .. controls (0,-0.05) .. (0.3,0.3);
    \draw (-0.3,-0.3) .. controls (0,0.05) .. (0.3,-0.3);
  }%
}
\begin{document}
	
	\title{The finiteness conjecture for skein modules}
	\address{Department of Mathematical Sciences, Montana State University, Bozeman, Montana, USA}
	\email{sam.gunningham@montana.edu}
	\author{Sam Gunningham}
	\address{School of Mathematics, University of Edinburgh, Edinburgh, UK}
	\email{D.Jordan@ed.ac.uk}
	\author{David Jordan}
	\address{Institut f\"{u}r Mathematik, Universit\"{a}t Z\"{u}rich, Zurich, Switzerland}
	\email{pavel.safronov@math.uzh.ch}
	\author{Pavel Safronov}
	
	\begin{abstract}
		We give a new, algebraically computable formula for skein modules of closed 3-manifolds via Heegaard splittings. As an application, we prove that skein modules of closed 3-manifolds are finite-dimensional, resolving in the affirmative a conjecture of Witten.
	\end{abstract}
	
	\maketitle
	
	\section*{Introduction}
	
	A fundamental invariant of an oriented 3-manifold $M$ emerging from quantum topology is its ``Kauffman bracket skein module" $\Sk(M)$ introduced by Przytycki \cite{Przytycki} and Turaev \cite{Turaev}. This is the $\Q(A)$-vector space formally spanned by all framed links in $M$, modulo isotopy equivalence and the linear relations,
	\begin{align*}
	\left\langle L \cup \KPA\right\rangle&=(-A^{2}-A^{-2})\langle L\rangle \\
	\left\langle\KPB\right\rangle &= A\left\langle\KPC\right\rangle + A^{-1} \left\langle \KPD \right\rangle,
	\end{align*}
	which are imposed between any links agreeing outside of some oriented 3-ball, and differing as depicted inside that ball. Despite the elementary definition, many basic properties of skein modules are not known. The main result of the present paper (\cref{thm:witten}) confirms a conjecture of Witten, and establishes the following most fundamental property of skein modules:
	
	\begin{introthm}
		The skein module of any closed oriented 3-manifold has finite dimension over $\Q(A)$.
		\label{thm:maintheorem}
	\end{introthm}
	
	Prior to Witten's conjecture, skein modules of closed 3-manifolds had been computed only for certain free quotients of $S^3$ by finite groups \cite{HostePrzytyckiLens, GilmerHarris}, surgeries on trefoil knots \cite{Bullock, Holmes} and a certain family of torus links \cite{Harris} (see the introduction of \cite{GilmerMasbaum} for more details).  Subsequently, Carrega \cite{Carrega} and Gilmer \cite{Gilmer} showed the skein module of the three-torus $T^3=S^1\times S^1 \times S^1$ to be 9-dimensional; Gilmer and Masbaum \cite{GilmerMasbaum} have established \emph{lower} bounds for dimensions of $\Sigma_g\times S^1$ for any genus\footnote{After the present paper first appeared on the arXiv, Detcherry and Wolff proved that Gilmer and Masbaum’s lower bound is the dimension \cite{Detcherry-Wolff}.}, and Detcherry \cite{Detcherry} has established the conjecture for surgeries along two-bridge and torus knots.
	
	\subsection*{Tensor product formula}
	
	We do not prove \cref{thm:maintheorem} by directly computing the dimensions (however, see \cref{sec:comp-alg}, ``Computer algebra").  Rather, \cref{thm:maintheorem} is one of a number of consequences of our second main theorem, which gives a new algebraic reformulation of skein modules, and brings to bear tools from the representation theory of quantum groups and deformation quantization modules.
	
	First let us remark that there are skein theories associated to any reductive algebraic group $G$ (indeed, for any ribbon category $\cA$), so we will phrase the results in this section in that generality; the Kauffman bracket skein module comes from $G=\SL_2$, more precisely from a standard choice of ribbon structure on the category $\Repq(SL_2)$.  We use the notation $\Sk_\cA$ for statements applying to a general braided tensor category, and the abbreviation  $\Sk_G$ for the case $\cA=\Repq(G)$ for a reductive group $G$.
	
	If $\Sigma$ is an oriented surface, the skein module $\Sk_\cA(\Sigma\times [0, 1])=\SkAlg_\cA(\Sigma)$ is a skein \emph{algebra}, where the composition is given by stacking skeins on top of each other.  Similarly, if $M$ is a 3-manifold with boundary $\Sigma$, then $\Sk_\cA(M)$ is naturally a module over $\SkAlg_\cA(\Sigma)$.  We begin by upgrading the skein algebra and skein module constructions to what we call the \emph{internal skein algebra} and \emph{internal skein module}: in the case of $G$-skein theories these are $\Uq(\g)$-equivariant algebras, and $\Uq(\g)$-equivariant modules over them, whose invariant part recovers the ordinary skein algebra and skein modules respectively.
	
	For this, let us pick a closed disk embedding $\DDc\hookrightarrow \Sigma$ and let $\Sigma^*=\Sigma\setminus \DDc$. The \emph{internal skein algebra} $\SkAlgint_\cA(\Sigma^*)$ (see \cref{def:intskeinalg}) is the algebra whose $V$-multiplicity space consists of skeins in $\Sigma^*\times [0, 1]$ which end at the boundary of $\Sigma^*\times\{0\}$ with label $V\in\cA$ (see \cref{fig:skein-stack}).  The usual skein algebra arises therefore as its invariant subalgebra, $\SkAlgint_\cA(\Sigma^*)^{\inv} = \Hom_\cA(\bu,\SkAlgint_\cA(\Sigma^*))$. We define the \emph{internal skein module} similarly (see \cref{def:intskeinmod}).
	
	Now suppose that $M$ is decomposed as $M=N_2\cup_{\Sigma}N_1$.  Our second main theorem (\cref{thm:relativetensorproduct}) is the following simple formula for the ordinary skein module of $M$ in terms of the internal skeins of its constituents.
	
	\begin{introthm}
		The natural evaluation pairing gives an isomorphism,
		\[\Sk_\cA(M)\cong \left(\Skint_\cA(N_2)\rt{\SkAlgint_\cA(\Sigma^*)} \Skint_\cA(N_1)\right)^{\inv}.\]
		\label{thm:tensorproduct}
	\end{introthm}
	
	Now let us present some important corollaries. Consider the case of a $G$-skein module for $q$ not a root of unity. In this case $\SkAlgint_G(\Sigma^*)$ is an algebra in the category $\Repq(G)$ of representations of the quantum group. This algebra coincides with the so-called Alekseev--Grosse--Schomerus algebra \cite{AlekseevGrosseSchomerus, BuffenoirRoche, RocheSzenes} which has more recently appeared in \cite{BZBJ1}. Using triviality of the M\"uger center of $\Repq(G)$ we prove (see \cref{cor:relativetensorproductMuger}) that the relative tensor product above is already invariant.
	
	\begin{introcor}
		Suppose $q$ is not a root of unity. Then the natural evaluation pairing gives an isomorphism,
		\[\Sk_G(M)\cong \Skint_G(N_2)\rt{\SkAlgint_G(\Sigma^*)} \Skint_G(N_1).\]
		\label{cor:tensorproduct}
	\end{introcor}
	
	We note that, in contrast to ordinary skein algebras, the internal skein algebra $\SkAlgint_G(\Sigma^*)$ has an explicit presentation in terms of generators and relations, so the above relative tensor product can be made quite explicit.  Moreover the internal skein algebras are \emph{smooth}, and in particular their limits as $q\to 1$ are smooth affine algebraic varieties, in contrast to the skein algebras which develop singularities.
	
	\begin{introrem}\label{rem:derived}
		\cref{cor:tensorproduct} suggests a potential definition for a derived analogue of the skein module of a $3$-manifold $M$ at generic $q$, namely, taking the derived tensor product
		\[
		L\Sk_G^\sharp(M) := \Skint_G(N_2)\rtL{\SkAlgint_G(\Sigma^*)} \Skint_G(N_1)
		\]
		The proof of \cref{thm:maintheorem} implies that this is a bounded chain complex whose homology groups are finite dimensional vector spaces over $\Q(q)$. The resulting homology groups is a skein-theoretic analogue of the sheaf-theoretic framed Floer homology $HP_\sharp^\bullet(M)$ introduced by Abouzaid and Manolescu \cite{AbouzaidManolescu} (see \cref{sec:complexified-instanton-floer-homology} for further details). A precise connection between the skein theoretic and sheaf theoretic invariants will be established in forthcoming work of the first and third named authors. The authors of this article intend to further study derived skein theoretic invariants in future work.
	\end{introrem}
	
	\begin{introex}
		Let us consider the simplest interesting example, the case $G=\SL_2$, and $\Sigma=T^2$.  Then the algebra $\SkAlgint_G(\Sigma^*)$ coincides with the so-called ``elliptic double" $\Dq(G)$, a subalgebra of the Heisenberg double of $\Uq(\sltwo)$ (see \cite{BrochierJordan} and for an expanded list of relations, \cite{BalagovicJordan}\daj{But the paper with Martina is GL2-specific rather than SL2.  Also Juliet's paper never wrote them out either, so they haven't appeared in print anywhere.  Could write an appendix to this paper giving the presentation for $\Dq(G)$ for all $G$?}).  This is the algebra generated by elements, $a^1_1, a^1_2, a^2_1, a^2_2, b^1_1, b^1_2, b^2_1, b^2_2$, subject to the relations,
		\[
		\begin{array}{ll}
		R_{21}A_1R_{12}A_2 &= A_2R_{21}A_1R_{12}\\
		R_{21}B_1RB_2 &= B_2R_{21}B_1R \\
		R_{21}B_1RA_2 &= A_2R_{21}B_1R_{21}^{-1}
		\end{array}, \qquad \textrm{ and } \quad \begin{array}{l}a^1_1a^2_2 - q^2 a^1_2 a^2_1 = 1,\\b^1_1b^2_2 - q^2 b^1_2 b^2_1 = 1\end{array}
		\]
		The first three equations take place in $\Dq(\SL_2)\otimes \End(V\otimes V)$, where
		\[A = \left(\begin{array}{cc} a^1_1 & a^1_2 \\ a^2_1 & a^2_2\end{array}\right),\qquad B = \left(\begin{array}{cc} b^1_1 & b^1_2 \\ b^2_1 & b^2_2 \end{array}\right),\qquad \begin{array}{l}A_1 = A\otimes \id,\qquad A_2 = \id\otimes A,\\ B_1 = B\otimes \id,\qquad B_2 = \id\otimes B,\end{array}\]
		and $R=R_{12}, R_{21}\in \End(V\otimes V)$ denote the quantum $R$-matrix and its flip, for the defining representation $V$ of $\Uq(\sltwo)$.
		
		The algebra $\Dq(G)$ may be regarded simultaneously as a deformation quantization of the variety $G\times G$ with its Heisenberg double Poisson structure \cite{STS}, and as a $q$-analogue of the algebra $\D(G)$ of differential operators on the group $G$.  The subalgebra of invariants in $\Dq(G)$ surjects onto the usual skein algebra of the torus, via a very general procedure known as quantum Hamiltonian reduction \cite{VaragnoloVasserot, BalagovicJordan, BZBJ2}.
	\end{introex}
	
	Typically, taking invariants does not commute with relative tensor products: the invariants in the tensor product are not spanned by the tensor product of the invariants in each factor.  However, in certain cases, when one of the factors is a \emph{cyclic} module over the internal skein algebra, we may in fact replace internal skein modules by ordinary skein modules in the formula (see \cref{prop:tensorproducthandlebody} and \cref{cor:connectedsum}).
	
	\begin{introcor}\label{cor:skalgtens}
		Suppose $q$ is not a root of unity and one of the following conditions is satisfied:
		\begin{itemize}
			\item $N_1$ and $N_2$ are handlebodies (hence define a Heegaard decomposition of $M$).
			
			\item $\Sigma = S^2$.
		\end{itemize}
		
		Then the natural evaluation pairing restricts to an isomorphism
		\[\Sk_G(M)\cong \Sk_G(N_2)\rt{\SkAlg_G(\Sigma)} \Sk_G(N_1).\]
	\end{introcor}
	
	We note that the skein algebra of $S^2$ is one-dimensional, so the case $\Sigma=S^2$ recovers the main theorem of \cite{PrzytyckiConnectedSum} expressing the skein module of a connected sum of three-manifolds as a tensor product of the skein modules. The case of a Heegaard splitting was also considered in \cite{McLendonTorsion}.  We would like to stress, however, that even in these cases, where one \emph{could} work directly with ordinary skein algebras, one perhaps should not:  the internal skein algebras are simply easier to work with for both proofs and computations.  In particular, it is difficult to present the skein module $\Sk_G(H^g)$ of the genus $g$ handlebody as a module for $\Sk_G(\Sigma_g)$, while by contrast $\Skint_G(H^g)$ is simply an induced module for $\SkAlgint_G(\Sigma_g)$.  Moreover, the failure of the classical character variety of the handlebody to define a smooth Lagrangian means that the deformation quantization techniques of \cite{KashiwaraSchapiraDQ} do not apply to skein algebras, while they do perfectly well for their internal enhancements.
	
	\subsection*{Proof of \texorpdfstring{\cref{thm:maintheorem}}{Theorem 1}} Let us now sketch our proof of \cref{thm:maintheorem} -- and its natural generalization to $G$-skein modules for any reductive group $G$ -- starting from \cref{thm:tensorproduct}.  The complete proof is given in \cref{sec:finite-dimensionality}.	
	
	A basic notion in the theory of ordinary differential equations on algebraic varieties is that of a \emph{holonomic system} -- this is a system of ``over-determined" differential equations, whose space of solutions is always finite-dimensional.  The algebra of polynomial differential operators on a smooth affine algebraic variety may be regarded as a deformation quantization of its cotangent space; in \cite{KashiwaraSchapiraDQ}, the notion of holonomicity was abstracted to hold for arbitrary deformation quantizations of smooth symplectic varieties besides cotangent spaces, and in this generality the same suite of finite-dimensionality results was established.  Because the internal skein algebras are flat deformations of smooth algebraic varieties, we may appeal to this deep and powerful general theory.
	
	Hence, given a closed 3-manifold $M$, we choose a Heegaard splitting $M=N_1\coprod_\Sigma N_2$, where $N_1$ and $N_2$ are handlebodies of genus $g$, and $\Sigma=\Sigma_g$ is their common boundary.  The internal skein algebra $\SkAlgint_G(\Sigma^*)$ is a deformation quantization (with the quantization parameter $q$) of the Poisson variety $G^{2g}$ with respect to the Fock--Rosly Poisson structure, which is generically symplectic.
	
	In \cref{thm:handlebody} we compute the handlebody modules $\Skint_G(N_1)$ and $\Skint_G(N_2)$ over the internal skein algebra $\SkAlgint_G(\Sigma^*)$ and show that they are also deformation quantizations, now of Lagrangian subvarieties $G^g\hookrightarrow G^{2g}$ (in particular, they lie in the symplectic locus). So, $\Skint_G(N_1)$ and $\Skint_G(N_2)$ determine \emph{holonomic} deformation quantization modules over the deformation quantization of $G^{2g}$.
	
	Appealing therefore to the theory of deformation quantization modules due to Kashiwara and Schapira \cite{KashiwaraSchapiraDQ}, we prove (see \cref{thm:holonomictensorproduct}) that the relative tensor product $\Skint_G(N_2)\otimes_{\SkAlgint_G(\Sigma^*)} \Skint_G(N_1)$ is finite-dimensional for $q=\exp(\hbar)$, where $\hbar$ is a formal parameter, and hence for generic $q$.  Using \cref{cor:tensorproduct} we identify the $G$-skein module of $M$ with the above relative tensor product, and the proof is complete.
	
	\subsection*{Further applications}
	
	One of the main tools in establishing \cref{thm:tensorproduct} is a construction of the skein TFT due to Walker \cite{Walker}. Namely, (see \cref{thm:WalkerTFT}) the assignment of a skein module $\Sk_\cA(M)$ to a closed 3-manifold $M$ and a \emph{skein category} $\SkCat_\cA(\Sigma)$ to a closed 2-manifold $\Sigma$ is a part of a topological field theory valued in categories and their bimodules.  Taking `free co-completions', we obtain a TFT $\rZ_\cA$ valued instead in locally presentable categories and their functors, which was shown in \cite{Cooke} to recover the factorization homology categories of \cite{BZBJ1}.
	
	It is a general feature of topological field theories that the value on $S^1\times X$ yields the corresponding categorical dimension of the value on $X$.  For a vector space, the categorical dimension is the ordinary dimension (an integer) while for a category, it is the categorical trace, or zeroth Hochschild homology (a vector space).  A corollary, \cref{lm:HHS1}, is that the skein module of $\Sigma\times S^1$ is identified with the Hochschild homology (a.k.a. categorical trace) of $\SkCat_\cA(\Sigma)$ (equivalently, of $\rZ_\cA(\Sigma)$).   Note that this property fails when one replaces skein categories by skein algebras: the Kauffman bracket skein module $\Sk(T^3)$ is $9$-dimensional, whereas the Hochschild homology of $\SkAlg(T^2)$ is $5$-dimensional \cite{Oblomkov,McLendonTraces}.
	
	Let us illustrate this perspective on two examples. Consider $\Sigma=S^2$. Then we give an equivalence (see \cref{prop:MugerS2}) between $\rZ_\cA(S^2)$ and the M\"uger center of $\cA$. In particular, we identify it with the trivial category of vector spaces, in the case of representations of the quantum group for $q$ not a root of unity.  Taking Hochschild homology we recover the result of \cite{HostePrzytycki} (see \cref{cor:skeinS2}).
	
	\begin{introcor}\label{cor:S2xS1}
		The $G$-skein module $\Sk_G(S^2\times S^1)$ is one-dimensional for $q$ not a root of unity.
	\end{introcor}
	
	Now consider $\Sigma=T^2$. In a forthcoming work of the first two authors with Monica Vazirani, we compute $\rZ_{\SL_N}(T^2)$ using a $q$-analogue of the generalized Springer decomposition \cite{Gunningham}. In the case $G=\SL_2$ it has the following description.
	
	\begin{introthm}[\cite{GJV}]
		We have a decomposition of abelian categories,
		\[\rZ_{\SL_2}(T^2) \simeq \lmod{\Dq(H)^{\Z_2}}{} \bigoplus \Vect \bigoplus \Vect \bigoplus \Vect \bigoplus \Vect.\]
	\end{introthm}
	
	Here $\lmod{\Dq(H)^{\Z_2}}{}$ is the ``Springer block'', where $H$ is the maximal torus of $\SL_2$, $\Dq(H)^{\Z_2}=\SkAlg(T^2)$ is the algebra of $\Z_2$-invariants on the quantum torus (see \cite{FrohmanGelca}), and the copies of $\Vect$ are four orthogonal ``cuspidal blocks'' which are supported at each of the four singular points $(\pm 1, \pm 1)$ of the $\Z_2$-action on $H \times H$. Taking Hochschild homology, and recalling that $\HH_0(\Dq(H)^{\Z_2})\cong \Q(A)^5$ \cite{Oblomkov,McLendonTraces}, we recover the computation,
	\[\Sk(T^3) \cong \HH_0(\rZ_{\SL_2}(T^2)) \cong \Q(A)^5 \oplus \Q(A)\oplus \Q(A)\oplus \Q(A)\oplus \Q(A),\]
	of \cite{Carrega, Gilmer} in a new way.  We expect it may be possible to compute $\Sk(\Sigma\times S^1)$ more generally using these techniques.   We discuss closely related TFTs, such as the Crane--Yetter and Kapustin--Witten TFTs, in \cref{sect:TFT}.
	
	Finally, let us remark that for simplicity we have restricted attention in the introduction on the case of $G$-skein modules defined over generic quantization parameters, since it is at that generality in which \cref{thm:maintheorem} holds, and since basic definitions in the root of unity case become more cumbersome.  However, we would like to stress that our results as formulated in the body of the paper also provide a systematic framework for studying the root of unity case, or more generally when we work over some arbitrary base ring such as $\Z[A,A^{-1}]$ in place of a field.
	
	In a previous work of Iordan Ganev and the latter two authors, \cite{GaJS}, we have formulated and proved a generalization of the ``Unicity conjecture" of Bonahon-Wong, for quantum $G$-character varieties of surfaces.  In future work, we intend to combine the techniques of the two papers to the study of torsion in skein modules of 3-manifolds at root-of-unity parameters, namely by lifting the constructions in the present paper to the relevant integral forms -- those coming from Temperley-Lieb diagrammatics, those coming from tilting modules, and those coming from Lusztig's divided powers quantum groups, and small quantum groups.  For example, \cref{thm:tensorproduct}, \cref{cor:tensorproduct}, \cref{cor:skalgtens}, and \cref{cor:S2xS1} all admit modifications, which involve structures such as Lusztig's quantum Frobenius homomorphism, which are special to the root of unity setting.
	
	\subsection*{Outline of the paper}
	
	In \cref{sect:algebra} we begin with the algebraic setup for the paper. We introduce some categorical notation and recall the basics of quantum groups and quantum moment maps. The latter notion allows us to discuss strongly equivariant modules and we prove that the relative tensor product of strongly equivariant modules lies in the M\"uger center (see \cref{prop:stronglyequivarianttensorproductMuger}). We finish the section by establishing a duality between left and right strongly equivariant modules.
	
	\Cref{sect:topology} is devoted to the skein-theoretic setup. We define skein modules and the skein category TFT for an arbitrary ribbon category and relate skein categories to factorization homology. We then introduce internal skein algebras and internal skein modules and compute them for surfaces (\cref{sec:skeinalgebrasofsurfaces}) and handlebodies (\cref{sect:handlebodies}) respectively.
	
	In \cref{sect:analysis} we discuss deformation quantization modules in the algebraic context. The main result there, \cref{thm:holonomictensorproduct}, establishes finite-dimensionality of the relative tensor product of two holonomic deformation quantization modules for a generic quantization parameter following Kashiwara and Schapira \cite{KashiwaraSchapiraDQ}.
	
	\Cref{sect:applications} collects all the ingredients from previous sections to prove theorems mentioned in the introduction. We prove a relative tensor product formula for skein modules (\cref{thm:relativetensorproduct}), relate the skein category of $S^2$ to the M\"uger center (\cref{prop:MugerS2}) and prove finite-dimensionality of $G$-skein modules of closed oriented 3-manifolds for generic parameters (\cref{thm:witten}).
	
	We end the paper with \cref{sect:discussion}, where we discuss how our results fit in the context of topological field theory, character theory and instanton Floer homology for complex groups and explain an approach for computing skein modules using computer algebra.
	
	\subsection*{Acknowledgments}
	We would like to thank David Ben-Zvi, Adrien Brochier, Juliet Cooke, Theo Johnson-Freyd, Yanki Lekili, Gregor Masbaum, Du Pei, Francois Petit, Peter Samuelson, Noah Snyder and Monica Vazirani for many valuable conversations realted to this work.  We are grateful to AIM San Jose, BIRS Oaxaca, CIRM Luminy, HIM Bonn, ICMS Edinburgh, QGM Aarhus, TSIMF Sanya, for hosting conferences where many of these conversations took place, and in particular to the ICMS Research in Groups program and the Aspen Center for Physics, for their hospitality during the writing stage.  The work of S.G. and D.J. was supported by the European Research Council (ERC) under the European Union’s Horizon 2020 research and innovation programme [grant agreement no. 637618]. The work of P.S. was supported by the NCCR SwissMAP grant of the Swiss National Science Foundation.
	
	\section{Algebra}
	\label{sect:algebra}
	
	This section treats the algebraic ingredients of our proof -- categories, quantum groups, quantum Harish-Chandra category, quantum moment maps and strongly equivariant modules.
	
	\subsection{Categories}
	
	We begin with some categorical preliminaries that will be used throughout the paper. In this section we work over an arbitrary commutative ring $\k$ which we will fix later.
	
	\begin{definition}
		The bicategory $\Cat$ has:
		\begin{itemize}
			\item As its objects small $\k$-linear categories.
			\item As the 1-morphisms from $\cC$ to $\cD$ the $\k$-linear functors $\cC\rightarrow \cD$.
			\item As the 2-morphisms natural transformations.
		\end{itemize}
	\end{definition}
	
	The bicategory $\Cat$ has a natural symmetric monoidal structure given by the tensor product of $\k$-linear categories.
	
	\begin{definition}
		Let $\cC\in\Cat$ be a small category. A \defterm{left $\cC$-module} is a functor $\cC^{\op}\rightarrow \Vect$. A \defterm{right $\cC$-module} is a functor $\cC\rightarrow \Vect$.
	\end{definition}
	
	\begin{remark}
		In the case when $\cC$ is a one-object category, the above two notions coincide with the notion of modules over the endomorphism algebra of the object of $\cC$.
	\end{remark}
	
	For many purposes the bicategory $\Cat$ does not have enough morphisms, and we require the following enlargement.
	
	\begin{definition}
		The bicategory $\Bimod$ has:
		\begin{itemize}
			\item As its objects small $\k$-linear categories.
			\item As the 1-morphisms from $\cC$ to $\cD$ the $\k$-linear functors $F\colon \cC\otimes \cD^{\op}\rightarrow \Vect$ (a.k.a. ``bimodules").
			\item As the 2-morphisms natural transformations.
		\end{itemize}
	\end{definition}
	
	The composition of $F\colon \cC\otimes \cD^{\op}\rightarrow \Vect$ and $G\colon \cD\otimes \cE^{\op}\rightarrow \Vect$ is the functor $F \otimes_\cD G\colon \cC\otimes \cE^{\op}\rightarrow \Vect$ given by the coend \cite[Chapter 7.8]{Borceux}:
	\[(F\otimes_\cD G)(c, e) = \int^{d\in\cD} F(c, d)\otimes G(d, e).\]
	Explicitly, it is given by the quotient
	\[\bigoplus_{d\in\cD} F(c, d)\otimes G(d, e)/\sim,\]
	where for any morphism $f\colon d''\rightarrow d'$ in $\cD$ we mod out by the image of
	\[F(c, d')\otimes G(d'', e)\xrightarrow{F(f)\otimes \id - \id\otimes G(f)} F(c, d'')\otimes G(d'', e)\oplus F(c, d')\otimes G(d', e).\]
	
	The tensor product of $\k$-linear categories equips $\Bimod$ with the structure of a symmetric monoidal bicategory \cite[Section 7]{DayStreet}.
	
	A typical category in $\Cat$ will not be closed under colimits -- for instance it may not admit direct sums or cokernels of morphisms.  We will therefore make occasional use of the notion of a locally presentable category -- this is a large category closed under arbitrary colimits, and satisfying some further set-theoretical conditions (we refer to \cite{BCJF} for complete definitions).
	
	\begin{definition}
		The bicategory $\PrL$ has:
		\begin{itemize}
			\item As its objects locally presentable $\k$-linear categories.
			\item As the 1-morphisms from $\cC$ to $\cD$ the cocontinuous functors $\cC\to\cD$.
			\item As the 2-morphisms the natural transformations.
		\end{itemize}
	\end{definition}
	
	The Kelly--Deligne tensor product equips $\PrL$ with the structure of a symmetric monoidal bicategory \cite[Chapter 5]{Bird}.
	
	We have symmetric monoidal functors
	\[\Cat\longrightarrow \Bimod\longrightarrow \PrL\]
	defined as follows:
	\begin{itemize}
		\item The functor $\Cat\longrightarrow \Bimod$ is the identity on objects and sends a functor $F\colon \cC\rightarrow \cD$ to the bimodule $\Hom_{\cD}(-, F(-))\colon \cC\times \cD^{\op}\rightarrow \Vect$.
		\item The functor $\fr{(-)}\colon\Bimod\rightarrow \PrL$ is the free cocompletion functor
		\[\fr{\cC} = \Fun(\cC^{\op}, \Vect).\]
		It is fully faithful and identifies $\Bimod$ with the full subcategory of $\PrL$ spanned by categories with enough compact projectives \footnote{Recall that an object $x$ in a locally presentable category $\cC$ is compact projective if the functor $\Hom_\cC(x,-)$ commutes with arbitrary colimits in $\cC$.}.
	\end{itemize}
	\begin{remark}The only locally presentable categories we will encounter are free cocompletions of small categories.\end{remark}
	
	Since $\fr{(-)}\colon \Cat\rightarrow \PrL$ is symmetric monoidal, it sends (braided) monoidal categories to (braided) monoidal categories. Suppose $\cC\in\Cat$ is a monoidal category and let  $F, G\in\fr{\cC}$. Then their tensor product is given by the Day convolution \cite{Day}
	\[(F\otimes G)(x) = \int^{y_1, y_2\in\cC} \Hom_\cC(x, y_1\otimes y_2) \otimes F(y_1)\otimes G(y_2).\]
	
	\begin{lemma}
		Suppose $\cC\in\Cat$ is a monoidal category. An algebra in $\fr{\cC}$ is the same as a lax monoidal functor $F\colon \cC^{\op}\rightarrow \Vect$.
		\label{lm:Dayalgebra}
	\end{lemma}
	
	Note also that if $F\colon \cC\rightarrow \cD$ is a morphism in $\Cat$, its image $F\colon \fr{\cC}\rightarrow \fr{\cD}$ is continuous and thus has a right adjoint $F^\R\colon \fr{\cD}\rightarrow \fr{\cC}$. Explicitly, the corresponding bifunctor $\fr{\cD}\times \cC^{\op}\rightarrow \Vect$ is given by $(P, x)\mapsto P(F(x))$.
	
	Using the symmetric monoidal structure on $\Cat$, $\Bimod$ and $\PrL$, we can talk about dualizable objects, i.e. categories $\cC$ equipped with a dual category $\cC^\vee$ and a pair of 1-morphisms $\ev\colon \cC^\vee\otimes \cC\rightarrow \bu$ and $\coev\colon \bu\rightarrow \cC\otimes \cC^\vee$ satisfying the usual duality axioms. Given a triple $(\cC, \cC^\vee, \ev)$, we say $\ev\colon \cC^\vee\otimes \cC\rightarrow \bu$ is a \defterm{nondegenerate pairing} if there is a coevaluation map exhibiting $\cC^\vee$ as the dual of $\cC$.
	
	\begin{example}
		Suppose $\cC\in\Bimod$. Then the pairing $\ev\colon \cC\otimes \cC^{\op}\rightarrow \Vect$ given by $x,y\mapsto \Hom_\cC(y, x)$ is a nondegenerate pairing in $\Bimod$. The corresponding coevaluation pairing is $\coev\colon \cC^{\op}\otimes \cC\rightarrow \Vect$ given by $x,y\mapsto \Hom_\cC(x, y)$.
		
		As a consequence, $\fr{\cC}\otimes \fr{\cC^{\op}}\rightarrow \Vect$ given by
		\[F, G\mapsto \int^{x\in\cC} F(x)\otimes G(x)\]
		is a nondegenerate pairing in $\PrL$.
		\label{ex:Homnondegenerate}
	\end{example}
	
	Using the notion of a dualizable category, we may introduce the notion of Hochschild homology.
	
	\begin{definition}
		Let $\cC\in\Cat$ be a category. Its \defterm{zeroth Hochschild homology} is
		\[\HH_0(\cC) = \int^{x\in\cC} \Hom_\cC(x, x)\in\Vect.\]
		\label{def:HH}
	\end{definition}
	
	\begin{remark}
		Recall from \cref{ex:Homnondegenerate} that every small category $\cC$ is dualizable in $\Bimod$. Then we may identify $\HH_0(\cC)\in\Vect$ as the composite $\ev\circ\coev$. Thus, the zeroth Hochschild homology of a category is an instance of the general notion of a dimension of a dualizable object (see e.g. \cite{BZNTraces}).
		\label{rmk:HHdimension}
	\end{remark}
	
	\subsection{Tensor products over categories and algebras}
	
	Suppose $\cC \in \Cat$ and let $F\colon \cC^{op} \to \Vect$ and $G\colon \cC\to \Vect$ be functors. In other words, $F$ is a left $\cC$-module and $G$ is a right $\cC$-module. Now suppose the category $\cC$ comes with a distinguished object $\bu\in \cC$. Then $F(\bu)$ is naturally a left module for $\End_\cC(\bu)$ and $G(\bu)$ is naturally a right module.
	
	In this section we will give some conditions for when $G\otimes_\cC F$ is given by the (ordinary) relative tensor product $G(\bu) \rt{\End_\cC(\bu)} F(\bu)$. 
	
	\begin{remark}
		The motivation for this section is the following. Suppose we have an oriented surface $\Sigma$ and a pair of oriented $3$-manifolds $N_0$ and $N_1$ together with isomorphisms $\partial \rev{N}_0 \cong \Sigma \cong \partial{N_1}$, where $\rev{N}_0$ refers to $N_0$ with the opposite orientation. In \cref{sect:relativetensorproduct} we will show that the skein module of $M = N_0 \cup_\Sigma N_1$ may be computed as the relative tensor product of certain functors over the skein category of $\Sigma$. We would like to understand the categorical conditions required for this tensor product to be computed as the relative tensor product of the skein modules of $N_0$ and $N_1$ over the skein algebra of $\Sigma$.
	\end{remark}
	
	\begin{definition}
		\label{def:cyclic}
		Let $\cC\in\Cat$ be a category together with a distinguished object $\bu\in\cC$ and let $F\colon \cC^{\op} \to \Vect$ be a left $\cC$-module. 
		
		\begin{enumerate}
			\item We say $F$ is \defterm{generated by invariants} if the morphism
			\begin{align*}
			F(\bu) \otimes \Hom(c,\bu) &\to F(c) \\
			s \otimes f &\mapsto F(f)(s)
			\end{align*}
			is surjective for every $c\in\cC$.
			
			\item We say $F$ is \defterm{cyclic} if there is an element $s_0\in F(\bu)$ such that the morphism
			\begin{align*}
			\Hom(c,\bu) &\to F(c) \\
			f &\mapsto F(f)(s_0)
			\end{align*}
			is surjective for every $c\in\cC$. In this case we say $F$ is \defterm{generated by $s_0$}.
		\end{enumerate}
	\end{definition}
	
	The definitions for right $\cC$-modules are given analogously.
	
	\begin{example}
		Suppose $A$ is a $\k$-algebra equipped with an action of a reductive algebraic group $G$. Let $\cC$ denote the category $\lmod{A}{\Rep(G)}^{\mathrm{cp}}$ of compact projective $G$-equivariant $A$-modules. Then we may identify the free cocompletion $\fr \cC$ with the category $\lmod{A}{\Rep(G)}$ of \emph{all} $G$-equivariant $A$-modules. Under this equivalence an object $M\in\lmod{A}{\Rep(G)}$ corresponds to the functor\daj{added the op below}
		\[
		F_M=\Hom(-,M)\colon (\lmod{A}{\Rep(G)}^{cp})^{op} \longrightarrow \Vect
		\]
		
		Note that $\cC$ is pointed by the object $A$ itself. Then $F_M$ is generated by invariants in the sense of \cref{def:cyclic} if and only if $M$ is generated as an $A$-module by $M^G$. Similarly, $F_M$ is cyclic if and only if there exists an element $s_0 \in M^G$ which generates $M$ as an $A$-module.
	\end{example}
	
	\begin{remark}
		The conditions of a module for a category being generated by invariants (respectively, being cyclic) correspond to natural skein theoretic conditions on the module over the skein category of a surface $\Sigma$ induced by a $3$-manifold bounding $\Sigma$ (see \cref{sect:handlebodies}).
	\end{remark}
	
	Let $\cC \in \Cat$ be a category with a distinguished object $\bu\in \cC$, and let 
	\[
	G\colon \cC \to \Vect, \qquad F\colon \cC^{op} \to \Vect
	\]
	be right and left $\cC$-modules respectively. Associated to this data is an algebra $\End_\cC(\bu)$ together with a right module $G(\bu)$ and a left module $F(\bu)$.
	
	Note that the embedding of the distinguished object $\bu$ into $\cC$ induces a map of relative tensor products
	\[G(\bu)\rt{\End_\cC(\bu)} F(\bu)\longrightarrow G\otimes_\cC F.\]
	
	\begin{proposition}
		Suppose $F$ and $G$ are generated by invariants. Then the map
		\[G(\bu)\rt{\End_\cC(\bu)} F(\bu)\longrightarrow G\rt{\cC} F\]
		is an isomorphism.
		\label{prop:reltensorproductcategory}
	\end{proposition}
	\begin{proof}
		Let $A = \End_\cC(\bu)$. Consider the diagram
		\[
		\xymatrix{
			\Hom_\cC(-, \bu)\otimes A\otimes F(\bu) \ar@<.5ex>[r] \ar@<-.5ex>[r] & \Hom_\cC(-, \bu)\otimes F(\bu) \ar[r] & F(-).
		}
		\]
		Since $F$ is generated by invariants, it is a coequalizer diagram. Similarly,
		\[
		\xymatrix{
			G(\bu)\otimes A\otimes \Hom_\cC(\bu, -) \ar@<.5ex>[r] \ar@<-.5ex>[r] & G(\bu)\otimes \Hom_\cC(\bu, -) \ar[r] & G(-)
		}
		\]
		is a coequalizer diagram as well.
		
		We have
		\[\int^{c\in\cC} G(\bu)\otimes \Hom_\cC(\bu, c)\otimes \Hom_\cC(c, \bu)\otimes F(\bu)\cong G(\bu)\otimes A\otimes F(\bu).\]
		Applying the above resolutions to $F$ and $G$, we get that $G\otimes_\cC F$ is computed as the colimit of
		\[
		\xymatrix{
			G(\bu)\otimes A\otimes A\otimes F(\bu) \ar@<.5ex>[r] \ar@<-.5ex>[r] & G(\bu)\otimes A\otimes F(\bu) \\
			& G(\bu)\otimes A\otimes A\otimes F(\bu) \ar@<.5ex>[u] \ar@<-.5ex>[u]
		}
		\]
		which computes
		\[G(\bu)\otimes_A A\otimes_A F(\bu)\cong G(\bu)\otimes_A F(\bu).\]
	\end{proof}
	
	\subsection{Quantum groups}
	
	Let $G$ be a connected reductive algebraic group and denote by $\Lambda$ and $\Lambda^\vee$ its weight and coweight lattices.   Let $\Uq(\g)$ be Lusztig's integral form of the quantum group defined over $\Z[q, q^{-1}]$ , see \cite{Lusztig}.  In particular, it has Cartan generators $K_\mu$ for $\mu$ in the coweight lattice $\Lambda^\vee$, and divided power Serre generators $E_i^{(r)}$ and $F_i^{(r)}$, for each simple root $\alpha_i$.  Fix a commutative ring $\k$, and a homomorphism $\Z[q,q^{-1}]\to k$. 
	
	\begin{definition}
		The category $\Repq(G)$ is the category of $\Lambda$-graded $\k$-modules $V=\oplus_{\lambda\in \Lambda} V_\lambda$ equipped with a compatible action of $\Uq(\g)$, i.e. such that $K_\mu v = q^{\langle \lambda, \mu\rangle} v$ for $v\in V_\lambda$, and such that for all $v\in V$, $E_i^{(r)}v=0$ and $F_i^{(s)}v=0$ for all but finitely many $r$ and $s$.
	\end{definition}
	
	The braiding and ribbon element on $\Repq(G)$ depend on further data in the ring $\k$.  For two simple roots $\alpha_i, \alpha_j\in \Lambda$ we denote by $\alpha_i\cdot \alpha_j\in\Z$ the $ij$ entry of the symmetrized Cartan matrix.  Choose\footnote{Such a $B$ exists and is unique, as soon as $d$ is divisible by the determinant of the Cartan matrix, so it is typical to fix that minimal choice for $d$, and suppress mention of $B$.} an integer $d$ and a symmetric bilinear form $B:\Lambda\times\Lambda\to \frac{1}{d}\mathbb{Z}$ such that $B(\alpha_i,\alpha_j)=\alpha_i\cdot\alpha_j$.  We henceforth fix a homomorphism $\Z[q^{1/d}, q^{-1/d}]\to\k$, so that we obtain a symmetric bilinear form $q^B\colon \Lambda\times \Lambda\rightarrow \k^\times,$ satisfying $q^B(\alpha_i, \alpha_j) = q^{\alpha_i\cdot \alpha_j}$.
	
	Using this bilinear form Lusztig \cite[Chapter 32]{Lusztig} equips $\Repq(G)$ with the structure of a $\k$-linear braided monoidal category. It is explained in \cite{SnyderTingley} that the additional choice of a homomorphism $\phi\colon \Lambda\rightarrow \Z/2$ such that $\phi(\alpha_i) = 0$ endows $\Repq(G)$ with a ribbon structure.
	
	\begin{remark}
		By convention, we will say $q$ \defterm{is generic} to mean $k=\Q(q^{1/d})$.  We will say $q$ \defterm{is not a root of unity} to mean either that $q$ is generic, or that $k=\C$ and $q^\ell\neq 1$ for all non-zero integers $\ell$.% either a number $q\in\C^\times$ which is not a root of unity in which case we set $\k=\C$ or $q\in\C(q^{1/d})=\k$.
	\end{remark}
	
	\begin{remark}
	We denote by $\Repqfd(G)\subset \Repq(G)$ the full subcategory of modules which are finitely generated over $k$.  For $q$ not a root of unity the category $\Repqfd(G)$ coincides with the full subcategory of compact projective objects in $\Repq(G)$. Note that it is not true at roots of unity as, for example, the trivial representation in that case is not projective.
	\end{remark}

	\begin{example}
		To fix the terminology, consider the case $G = \SL_2$. For $q$ not a root of unity the category $\Repq(\SL_2)$ has simple objects $V(m)$, for each $m \in \mathbb{N}_0$, the irreducible highest weight representation of $\Uq(\mathfrak{sl}_2)$ of highest weight $m$ and of dimension $m+1$.  Every object of $\Repq(\SL_2)$ is a (possibly infinite) direct sum of simple modules, while $\Repqfd(SL_2)$ consists of finite direct sums of simple objects.  In this case, we take $d=2$, and the braiding $\sigma\colon V(1)\otimes V(1)\rightarrow V(1)\otimes V(1)$ is given
		by applying the $R$-matrix
		\[
		R = q^{-1/2}\left(\begin{array}{cccc}
		q & 0 & 0 & 0 \\
		0 & 1 & 0 & 0 \\
		0 & q - q^{-1} & 1 & 0 \\
		0 & 0 & 0 & q
		\end{array}\right),
		\]
		post-composed with the tensor flip.  The category $\Repq(\PSL_2)$ is the full subcategory of $\Repq(\SL_2)$ generated by $V(m)$ for $m$ even.
	\end{example}

	Recall that the M\"{u}ger center $\ZMug(\cA)$ of a braided monoidal category $\cA$ is the full subcategory consisting of objects $x\in\cA$ such that for every $y\in\cA$ the map $\sigma_{y, x}\circ\sigma_{x, y}$ is the identity. We will say the M\"uger center of $\cA$ is trivial if every object in $\ZMug(\cA)$ is a direct sum of the unit object.  The following is well-known:
	
	\begin{proposition}
		Suppose $q$ is not a root of unity. Then the M\"uger center of $\Repq(G)$ is trivial.
		\label{prop:genericMugercenter}
	\end{proposition}
	
	\subsection{Braided function algebra}
	
	Fix the ground ring $\k$ and a ribbon $\k$-linear category $\cA$ (we assume that the unit $\bu$ is simple). Let $\cA^{\bop}$ be the same monoidal category as $\cA$ with the braiding given by $\sigma_{y, x}^{-1}\colon x\otimes y\rightarrow y\otimes x$. The free cocompletion $\fr{\cA}$ inherits a braided monoidal structure from $\cA$ given by the Day convolution.
	
	Let $T\colon \cA\otimes \cA^{\bop}\rightarrow \cA$ be the tensor product functor. After passing to free cocompletions it admits a right adjoint $T^\R\colon \cA\rightarrow \fr{\cA}\otimes \fr{\cA}$. The following definition goes back to the works \cite{Majid, Lyubashenko}.
	
	\begin{definition}
		The \defterm{braided function algebra $\cF$} is
		\[\cF = T(T^\R(\bu))\in\fr{\cA}.\]
	\end{definition}
	
	Explicitly, we may identify $\cF$ as the colimit
	\begin{align*}
	\cF &\cong \underset{V,W\in\cA, f\colon V\otimes W\rightarrow \bu}{\colim} V\otimes W \\
	&\cong \int^{X\in \cA} X^*\otimes X.
	\end{align*}
	
	Since $T^\R$ is lax monoidal, $\cF$ is naturally an algebra in $\fr{\cA}$. Moreover, since $TT^\R$ is a comonad, $\cF$ naturally becomes a bialgebra in $\fr{\cA}$. We denote by $\epsilon\colon \cF\rightarrow \bu$ the counit of $\cF$.
	
	\begin{example}
		Let $\cA=\Repqfd(G)$ for $q$ not a root of unity. Then the Peter--Weyl theorem gives
		\[\cF \cong \bigoplus_{\lambda\in\Lambda^{dom}} V(\lambda)^*\otimes V(\lambda),\]
		where $\Lambda^{dom}$ is the set of dominant weights.
	\end{example}
	
	\begin{example}
		More concretely, let $\cA = \Repqfd(\SL_2)$. Then $\Oq(\SL_2) = \cF$ can be presented with generators $a^1_1, a^1_2, a^2_1, a^2_2$, and relations,
		\[
		R_{21}A_1R_{12}A_2 = A_2R_{21}A_1R_{12}, \qquad a^1_1a^2_2 - q^2 a^1_2 a^2_1 = 1.
		\]
		The first equation takes place in $\Oq(\SL_2)\otimes \End(V\otimes V)$, where
		\[A = \left(\begin{array}{cc} a^1_1 & a^1_2 \\ a^2_1 & a^2_2 \end{array}\right),\qquad A_1 = A\otimes \id,\qquad A_2 = \id\otimes A,\]
		and $V$ denote the defining two-dimensional representation of $\Uq(\sltwo)$.
		These may be expanded out explicitly as:
		\begin{align*}
		\begin{aligned}\label{relns-aa}
		&a^1_2a^1_1 = a^1_1a^1_2 + (1-q^{-2})a^1_2a^2_2 & \quad  &a^2_2a^1_1 = a^1_1a^2_2& &\\
		&a^2_1a^1_1 = a^1_1a^2_1 - (1-q^{-2})a^2_2a^2_1& \quad  &a^2_2a^1_2 = q^2a^1_2a^2_2& &\\
		&a^2_1a^1_2 = a^1_2a^2_1 + (1-q^{-2})(a^1_1a^2_2 - a^2_2 a^2_2)& \quad  &a^2_2a^2_1 = q^{-2}a^2_1a^2_2\\
		\end{aligned}
		\end{align*}
	\end{example}
	
	\subsection{Harish-Chandra category}
	\label{sect:HC}
	
	Let $\rZ(\fr{\cA})$ be the Drinfeld center of the monoidal category $\fr{\cA}$. Since $\fr{\cA}$ is braided, we have a natural braided monoidal functor
	\[\fr{\cA}\otimes\fr{\cA}^{\bop}\longrightarrow \rZ(\fr{\cA})\]
	given by the left and right action of $\fr{\cA}$ on itself. In particular, for every pair of objects $x\in\fr{\cA}\otimes \fr{\cA}^{\bop}$ and $V\in\fr{\cA}$ we have a natural isomorphism
	\[V\otimes T(x)\longrightarrow T(x)\otimes V.\]
	
	For instance, for $\cF = T(T^\R(\bu))$ we obtain the \defterm{field goal transform}
	\[\tau_V\colon V\otimes \cF\longrightarrow \cF\otimes V.\]
	
	Explicitly, in terms of the coend components $X^*\otimes X$ of $\cF$, the map $\tau_V$ is given by
	\[V \otimes X^*\otimes X \xrightarrow{\sigma_{V, X}\circ\sigma^{-1}_{X^*,V}} X^*\otimes X\otimes V,\]
	see \cref{fig:fieldgoal}.
	
	\begin{figure}
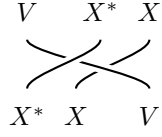

		\tik{
			\node[label=above:$V$] at (0,0) {};
			\node[label=above:$X^*$] at (1.5,0) {};
			\node[label=above:$X$] at (2.5,0) {};
			
			\node[label=below:$V$] at (2.5,-1) {};
			\node[label=below:$X^*$] at (0,-1) {};
			\node[label=below:$X$] at (1,-1) {};
			
			\fieldgoal{0}{0};
		}
		\caption{The field goal transform.}
		\label{fig:fieldgoal}
	\end{figure}
	
	\begin{definition}
		The \defterm{Harish-Chandra category} is
		\[\HC(\cA) = \lmod{\cF}{\fr{\cA}}.\]
	\end{definition}
	
	\begin{remark}
		Applied to the case $\cA=\Repqfd(G)$, we obtain a quantum group analogue of the category of $\U\g$-bimodules whose diagonal action is integrable, i.e. the category of Harish-Chandra bimodules. We refer to \cite{SafronovQMM} for more on this perspective.
	\end{remark}
	
	Since $T^\R(\bu)\in\fr{\cA}\otimes \fr{\cA}^{\bop}$ is a commutative algebra, $\HC(\cA)$ carries a natural monoidal structure given as follows. The field goal transform provides an identification
	\[
	\tau_{lr}\colon \lmod{\cF}{\fr{\cA}}\xrightarrow{\sim} \rmod{\cF}{\fr{\cA}}
	\]
	and the monoidal structure on $\HC(\cA)$ is given by the relative tensor product over $\cF$.
	
	In addition to the monoidal structure, the Harish-Chandra category possesses the following algebraic structures:
	\begin{itemize}
		\item A monoidal functor $\fr{\cA}\rightarrow \HC(\cA)$ given by the free left $\cF$-module $V\mapsto \cF\otimes V$. Equivalently, it is given by the free right $\cF$-module $V\otimes \cF$.
		
		\item A functor $\coinv_l\colon \HC(\cA)\rightarrow \fr{\cA}$ given by coinvariants on the left $M\mapsto \bu\rt{\cF} M$. Its right adjoint $\triv_l\colon \fr{\cA}\rightarrow \HC(\cA)$ is given by sending $V\in\fr{\cA}$ to the trivial left $\cF$-module.
		
		\item A functor $\coinv_r\colon \HC(\cA)\rightarrow \fr{\cA}$ given by coinvariants on the right $M\mapsto M\rt{\cF} \bu$. Its right adjoint $\triv_r\colon \fr{\cA}\rightarrow \HC(\cA)$ is given by sending $V\in\fr{\cA}$ to the trivial right $\cF$-module.
	\end{itemize}
	
	\begin{lemma}
		The functors $\triv_l, \triv_r\colon \fr{\cA}\rightarrow \HC(\cA)$ are fully faithful.
	\end{lemma}
	\begin{proof}
		The counit of the adjunction $\coinv_l\dashv \triv_l$ is $\bu\rt{\cF} \triv_l(V)\rightarrow V$ which is an isomorphism.
	\end{proof}
	
	The category $\fr{\cA}$ carries two $\HC(\cA)$-module structures:
	\begin{itemize}
		\item The functor $M\in\HC(\cA), V\in\fr{\cA}\mapsto M\rt{\cF}(\cF\otimes V) \rt{\cF} \bu$ equips $\fr{\cA}$ with a left $\HC(\cA)$-module structure. With respect to it $\coinv_r\colon \HC(\cA)\rightarrow \fr{\cA}$ becomes a functor of left $\HC(\cA)$-module categories.
		
		\item The functor $V\in\fr{\cA}, M\in\HC(\cA)\mapsto \bu\rt{\cF}(\cF\otimes V)\rt{\cF} M$ equips $\fr{\cA}$ with a right $\HC(\cA)$-module structure. With respect to it $\coinv_l\colon \HC(\cA)\rightarrow \fr{\cA}$ becomes a functor of right $\HC(\cA)$-module categories.
	\end{itemize}
	
	\begin{proposition}
		Suppose for $V\in\fr{\cA}$. The object $\triv_r(V)\in\HC(\cA)$ has the trivial left $\cF$-module structure iff $V$ lies in the M\"uger center of $\fr{\cA}$.
		\label{prop:trivialOqGbimodule}
	\end{proposition}
	\begin{proof}
		Let $X\in\cA$. By rigidity of $\cA$ we may identify
		\[\Hom_{\fr{\cA}}(X\otimes V, X\otimes V)\cong\Hom_{\fr{\cA}}(X^*\otimes X\otimes V, V).\]
		
		Under the isomorphism the image of $\sigma_{V, X}\circ \sigma_{X, V}$ is the composite
		\[
		X^*\otimes X\otimes V\xrightarrow{\id\otimes \sigma_{X, V}}X^*\otimes V\otimes X\xrightarrow{\sigma^{-1}_{V, X^*}\otimes \id}V\otimes X^*\otimes X\xrightarrow{\id\otimes \ev} V
		\]
		which coincides with the left action of $\cF$ on $\triv_r(V)$. The image of $\id\colon X\otimes V\rightarrow X\otimes V$ is
		\[
		X^*\otimes X\otimes V\xrightarrow{\ev\otimes \id} V
		\]
		which coincides with the trivial left action of $\cF$.
		
		Thus, $\sigma_{V, X}\circ\sigma_{X, V} = \id_{X\otimes V}$ for every $X\in\cA$ iff $\triv_r(V)$ has the trivial left $\cF$-module structure.
	\end{proof}
	
	\begin{corollary}
		For any $V,W\in\fr{\cA}$ we have
		\[\triv_l(V)\rt{\cF} \triv_r(W)\in \ZMug(\fr{\cA}).\]
		\label{cor:OqGtensorproducttrivial}
	\end{corollary}
	\begin{proof}
		Consider the object $\triv_l(V)\rt{\cF} \triv_r(W)\in\HC(\cA)$. As an object of $\rmod{\cF}{\fr{\cA}}\cong \HC(\cA)$, it has a trivial right $\cF$-action, i.e. it lies in the image of $\triv_r$. Similarly, as an object of $\lmod{\cF}{\fr{\cA}}\cong \HC(\cA)$, it has a trivial left $\cF$-action, i.e. it lies in the image of $\triv_l$. Therefore, by \cref{prop:trivialOqGbimodule} it lies in the M\"uger center.
	\end{proof}
	
	\subsection{Quantum moment maps}
	
	Recall that $T^\R(\bu)$ is a commutative algebra in $\fr{\cA}\otimes \fr{\cA}^{\bop}$ and $\cF=T(T^\R(\bu))$ is its image in $\fr{\cA}$.
	
	\begin{definition}
		Let $A$ be an algebra in $\fr{\cA}$. A \defterm{quantum moment map} is an algebra map $\mu\colon \cF\rightarrow A$ in $\fr{\cA}$ whose adjoint $T^\R(\bu)\rightarrow T^\R(A)$ is a central map.
	\end{definition}
	
	\begin{remark}
		Equivalently, the quantum moment map equation can be formulated as a commutativity of the diagram
		\[
		\xymatrix{
			A\otimes \cF \ar^{\id\otimes \mu}[r] \ar^{\tau_A}[dd] & A\otimes A \ar^{m}[dr] & \\
			&& A \\
			\cF\otimes A \ar^{\mu\otimes \id}[r] & A\otimes A \ar_{m}[ur] &
		}
		\]
	\end{remark}
	
	\begin{remark}
		Let us remark that there are several closely related constructions throughout the literature going by the name ``quantum moment map''. Quantizations of moment maps $M\rightarrow \g^*$ on Poisson manifolds with a Hamiltonian $\g$-action are given by homomorphisms $U\g\rightarrow A$. If $G$ is a Poisson-Lie group and $G^*$ its dual, one can also consider $G$-actions on $M$ with a moment map $M\rightarrow G^*$. In this setting quantum moment maps are given by homomorphisms $H\rightarrow A$, where $H$ is the Hopf algebra quantizing $G^*$ \cite{Lu}. In our setting we are interested in, on the classical level, actions of Poisson-Lie groups $G$ on $M$ with a moment map $M\rightarrow G$ (where the target is equipped with the so-called Semenov-Tian-Shansky Poisson structure) and, on the quantum level, quantum moment maps $\cF\rightarrow A$, where $\cF$ is an $H$-comodule algebra. These (quantum) moment maps were first considered in \cite{VaragnoloVasserot} and extensively used in \cite{BZBJ2} to describe factorization homology of closed surfaces. We refer to \cite{SafronovQMM} for more details on this definition of quantum moment maps and how it reduces to the previous definitions.
	\end{remark}
	
	The following is shown in \cite[Corollary 4.7]{BZBJ2}.
	
	\begin{proposition}
		Suppose $A\in\fr{\cA}$ is an algebra. The right action of $\HC(\cA)$ on $\lmod{A}{\fr{\cA}}$ compatible with the natural right $\cA$-module structure on $\lmod{A}{\fr{\cA}}$ is the same as the data of a quantum moment map $\mu\colon \cF\rightarrow A$.
		\label{prop:HCmomentmap}
	\end{proposition}
	
	Explicitly, suppose $A\in\fr{\cA}$ carries a quantum moment map $\mu\colon \cF\rightarrow A$. Given a left $A$-module $V$ and a left $\cF$-module $M$, the action is
	\[\tau_{lr}(\mu_*(V))\rt{\cF} M,\]
	where $\mu_*(V)$ is $V$ considered as a left $\cF$-module via the quantum moment map.
	
	Quantum moment maps allow us to introduce the notion of strongly equivariant modules.
	
	\begin{definition}
		Let $A\in\fr{\cA}$ be an algebra equipped with a quantum moment map $\mu\colon \cF\rightarrow A$. A left $A$-module $V$ is \defterm{strongly equivariant} if $\tau_{lr}(\mu_*(V))$ is a trivial right $\cF$-module. We denote by
		\[\lmod{A}{\fr{\cA}}^\str\subset \lmod{A}{\fr{\cA}}\]
		the full subcategory of strongly equivariant modules.
	\end{definition}
	
	\begin{remark}
		Explicitly, a left $A$-module $V$ is strongly equivariant if the diagram
		\[
		\xymatrix{
			& V\otimes \cF \ar^{\id\otimes \epsilon}[dr] \ar_{\tau_V}[dl] & \\
			\cF \otimes V \ar_{\mu\otimes \id}[r] & A\otimes V \ar_{\act_V}[r] & V
		}
		\]
		is commutative.
	\end{remark}
	
	\begin{remark}
		Let $X$ be a smooth affine variety equipped with an action of a algebraic group $G$. Let $\D(X)$ be the algebra of global differential operators on $X$. It carries a moment map $\mu\colon \U\g\rightarrow \D(X)$ given by the action vector fields. Recall that a weakly equivariant $\D$-module on $X$ is an object of $\lmod{\D(X)}{\Rep(G)}$. If $G$ is connected, then strongly equivariant $\D$-modules, i.e. $\D$-modules on the stack $[X/G]$, form a full subcategory of weakly equivariant $\D$-modules $M$ where the $\U\g$-action induced by the moment map coincides with the $\U\g$-action coming from the $G$-action on $M$.
	\end{remark}
	
	We will also use the following perspective on the strongly equivariant category. Let $A\in\fr{\cA}$ be an algebra equipped with a quantum moment map. There is a monad on $\lmod{A}{\fr{\cA}}$ given by the composition
	\begin{equation}
	S\colon \lmod{A}{\fr{\cA}}\longrightarrow \bimod{A}{\cF}{\fr{\cA}}\xrightarrow{\coinv_r}\lmod{A}{\fr{\cA}},
	\label{eq:stronglyequivariantmonad}
	\end{equation}
	where the first functor turns an $A$-module into an $(A, \cF)$-bimodule with the right $\cF$-module structure coming from the quantum moment map. This monad is idempotent, so the forgetful functor from $S$-algebras in $\lmod{A}{\fr{\cA}}$ to $\lmod{A}{\fr{\cA}}$ is fully faithful.
	
	Recall that $\fr{\cA}$ carries a natural left $\HC(\cA)$-action. The following statement is proved in \cite[Theorem 5.2]{BZBJ2}.
	
	\begin{proposition}
		Let $A\in\fr{\cA}$ be an algebra equipped with a quantum moment map. There is an equivalence of categories
		\[\lmod{A}{\fr{\cA}}^\str\cong\lmod{A}{\fr{\cA}}\rt{\HC(\cA)} \fr{\cA}.\]
		\label{prop:HCstrongequivariance}
	\end{proposition}
	
	\begin{remark}
		One way to see \cref{prop:HCstrongequivariance} is as follows. The relative tensor product $\lmod{A}{\fr{\cA}}\rt{\HC(\cA)} \fr{\cA}$ is obtained as the geometric realization of the simplicial object
		\[
		\xymatrix{
			{\lmod{A}{\fr{\cA}}} & {\lmod{A}{\fr{\cA}}\otimes_{\fr{\cA}} \HC(\cA) \cong\bimod{A}{\cF}{\fr{\cA}}} \ar@<.5ex>[l] \ar@<-.5ex>[l] & \ldots \ar@<.8ex>[l] \ar[l] \ar@<-.8ex>[l]
		}
		\]
		in $\PrL$. Since $\HC(\cA)$ is rigid, this diagram admits right adjoints which satisfy the Beck--Chevalley conditions. Therefore, by \cite[Theorem 4.7.5.2]{HA} the right adjoint to the projection
		\[\lmod{A}{\fr{\cA}}\longrightarrow \lmod{A}{\fr{\cA}}\rt{\HC(\cA)} \fr{\cA}\]
		is monadic and the monad is identified with the monad $S$ introduced above.
	\end{remark}
	
	Note that all statements about left modules have a symmetric counterpart for right modules, so that we can define strongly equivariant right $A$-modules with an equivalence
	\[\rmod{A}{\fr{\cA}}^\str\cong \fr{\cA}\rt{\HC(\cA)} \rmod{A}{\fr{\cA}}.\]
	
	\begin{proposition}
		For any two objects $V\in \rmod{A}{\fr{\cA}}^\str$ and $W\in\lmod{A}{\fr{\cA}}^\str$ we have
		\[V\rt{A} W\in \ZMug(\fr{\cA}).\]
		\label{prop:stronglyequivarianttensorproductMuger}
	\end{proposition}
	\begin{proof}
		We have an epimorphism
		\[V\rt{\cF} W\longrightarrow V\rt{A} W\]
		where we consider $V$ and $W$ as $\cF$-modules via the quantum moment map. Due to strong equivariance, we may identify
		\[V\rt{\cF} W\cong \triv_l(V)\rt{\cF} \triv_r(W),\]
		so by \cref{cor:OqGtensorproducttrivial} $V\rt{A} W$ is a quotient of an object in the M\"uger center, therefore it lies in the M\"uger center itself.
	\end{proof}
	
	\subsection{Duality and strong equivariance}
	
	In this section we establish a duality property for the category of modules over algebras equipped with a quantum moment map. In this section $A\in\fr{\cA}$ is an algebra equipped with a quantum moment map $\mu\colon \cF\rightarrow A$.
	
	\begin{proposition}
		The functor
		\[\ev\colon \rmod{A}{\fr{\cA}}\otimes \lmod{A}{\fr{\cA}}\rightarrow \Vect\]
		given by $\ev(M, N) = \Hom_{\fr{\cA}}(\bu, M\rt{A} N)$ is a nondegenerate pairing in $\PrL$.
		\label{prop:weaklyequivariantduality}
	\end{proposition}
	\begin{proof}
		Since the unit $\bu\in\fr{\cA}$ is compact and projective, $\ev$ is a colimit-preserving functor. In the proof all modules and bimodules are considered internal to $\fr{\cA}$.
		
		Consider the functor
		\begin{align*}
		\mu\colon \lmod{A}{}\otimes \rmod{A}{}&\longrightarrow \lmod{A}{}\rt{\fr{\cA}} \rmod{A}{}\\
		&\cong \bimod{A}{A}{}.
		\end{align*}
		By \cite[Proposition 3.17]{BZBJ1} it admits a colimit-preserving right adjoint $\mu^\R$. We may therefore define the coevaluation map to be
		\[\Vect\xrightarrow{A} \bimod{A}{A}{}\xrightarrow{\mu^R}\lmod{A}{}\otimes \rmod{A}{},\]
		where the first functor sends a vector space $V$ to the $(A, A)$-bimodule $V\otimes A$.
		
		The duality axioms follow from the commutative diagram
		\[
		\xymatrix{
			{\lmod{A}{} \otimes \bimod{A}{A}{}}\ar^-{\id\otimes \mu^\R}[r] \ar^{\mu}[d] & {\lmod{A}{}\otimes \rmod{A}{}\otimes \lmod{A}{}} \ar^{\mu\otimes \id}[d] \\
			{\bimod{A\otimes A}{A}{}} \ar^{\mu^\R}[r] & {\bimod{A}{A}{}\otimes \lmod{A}{}}
		}
		\]
		and similarly for $\rmod{A}{}$ which in turn follow from the fact that by rigidity $T^\R\colon \cA\rightarrow \fr{\cA}\otimes \fr{\cA}$ is a functor of $(\cA, \cA)$-bimodule categories.
	\end{proof}
	
	We will now construct a duality pairing for the strongly equivariant category.
	
	\begin{proposition}
		The functor
		\[\rmod{A}{\fr{\cA}}^\str\otimes \lmod{A}{\fr{\cA}}^\str\longrightarrow \Vect\]
		given by $M,N\mapsto \Hom_{\fr{\cA}}(\bu, M\rt{A} N)$ is a nondegenerate pairing.
		\label{prop:stronglyequivariantduality}
	\end{proposition}
	\begin{proof}
		Let $\coev\colon \Vect\rightarrow \lmod{A}{\fr{\cA}}\otimes \rmod{A}{\fr{\cA}}$ be the coevaluation pairing constructed in \cref{prop:weaklyequivariantduality}.
		
		Recall the monad $S\colon \lmod{A}{\fr{\cA}}\rightarrow \lmod{A}{\fr{\cA}}$. Since it is given by taking coinvariants, it is clearly colimit-preserving. In particular, it makes sense to consider the dual monad $S^\vee\colon \rmod{A}{\fr{\cA}}\rightarrow \rmod{A}{\fr{\cA}}$. Consider $M\in \rmod{A}{\fr{\cA}}$ and $N\in\lmod{A}{\fr{\cA}}$. We have a natural isomorphism
		\[\bu\rt{\cF} M\rt{A} N\cong M\rt{A} N\rt{\cF} \bu\]
		which identifies $S^\vee$ with the monad on $\rmod{A}{\fr{\cA}}$ whose algebras are strongly equivariant right $A$-modules.
		
		We define the coevaluation pairing on the strongly equivariant category to be given by the composite
		\[\Vect\xrightarrow{\coev}\lmod{A}{\fr{\cA}}\otimes \rmod{A}{\fr{\cA}}\xrightarrow{S\otimes S^\vee} \lmod{A}{\fr{\cA}}^{\str}\otimes \rmod{A}{\fr{\cA}}^{\str}.\]
		Note that since $S$ is idempotent, it is equivalent to $(S\otimes \id)\circ\coev\cong (\id\otimes S^\vee)\circ\coev$.
		
		Using the relation $\ev\circ (S^\vee\otimes \id)\cong \ev\circ (\id\otimes S)$, the duality axioms for $\lmod{A}{\fr{\cA}}^{\str}$ reduce to those for $\lmod{A}{\fr{\cA}}$.
	\end{proof}
	
	\section{Topology}
	\label{sect:topology}
	
	This section treats the topological ingredients of our proof -- Walker's skein category TFT, its relation to factorization homology, monadic reconstruction of factorization homology, and finally reconstruction for handlebodies.
	
	Throughout this section we fix $\cA\in\Cat$, a ribbon category linear over some ring $k$ whose unit $\bu\in\cA$ is simple. We denote by $\DD\subset \R^2$ the open unit disk and $\DDc$ the standard closed disk.
	
	\subsection{The skein category TFT}
	
	A fundamental ingredient in the proof of \cref{thm:maintheorem} is an idea due to Kevin Walker: we can enhance the skein module invariants of 3-manifolds to a (3,2)-dimensional TFT in the Atiyah--Segal framework by assigning to a surface $\Sigma$ the ``skein category" $\SkCat(\Sigma)$, and to a 3-manifold $M$ with boundary $\partial M = \rev{\Sigma}_{in}\sqcup \Sigma_{out}$, a categorical $(\SkCat(\Sigma_{in}), \SkCat(\Sigma_{out}))$ ``skein bimodule'' $\Sk(M)$.  We recall these constructions now.
	
	Let us sketch the definition of a skein category of a surface \cite{Walker,JF,Cooke}.
	
	\begin{definition}[Sketch.  See \cref{fig:ribbon} and {\cite[Section 4.2]{Cooke}}]
		Let $\Sigma$ be an oriented surface.
		\begin{itemize}
			\item An \defterm{$\cA$-labeling of $\Sigma$} is the data, $X$, of an oriented embedding of finitely many disjoint disks $x_1, \dots, x_n\colon \DD\rightarrow \Sigma$ labeled by objects $V_1, \dots, V_n$ of $\cA$.  We denote by $\vec{x_i}$ the $x$ axis sitting inside each disk $x_i$, and denote $\vec{X}=\cup_i\vec{x_i}$.
			\item A ribbon graph has ``ribbons" connecting ``coupons".  As topological spaces, ribbons and coupons are simply embedded rectangles $I\times I$, however, we require that ribbons begin and end at either the top ``outgoing", or bottom ``incoming", boundary interval of some coupon, or else at $\Sigma\times\{0,1\}$.
			\item An \defterm{$\cA$-coloring} of a ribbon graph is a labelling of each ribbon by an object of $\cA$, and of each coupon by a morphism from the (ordered) tensor product of incoming edges to the (ordered) tensor product of outgoing edges.
			\item We say that an $\cA$-colored ribbon graph $\Gamma$ is \defterm{compatible with an $\cA$-labeling} if $\partial \Gamma=\vec{X}$, and denote by $\mathrm{Rib}_\cA(M,X)$ the $\k$-vector space with basis the $\cA$-colored ribbon graphs on $M$ compatible with $X$.
		\end{itemize}
	\end{definition}
	
	\begin{figure}
		\includegraphics[height=1.25in]{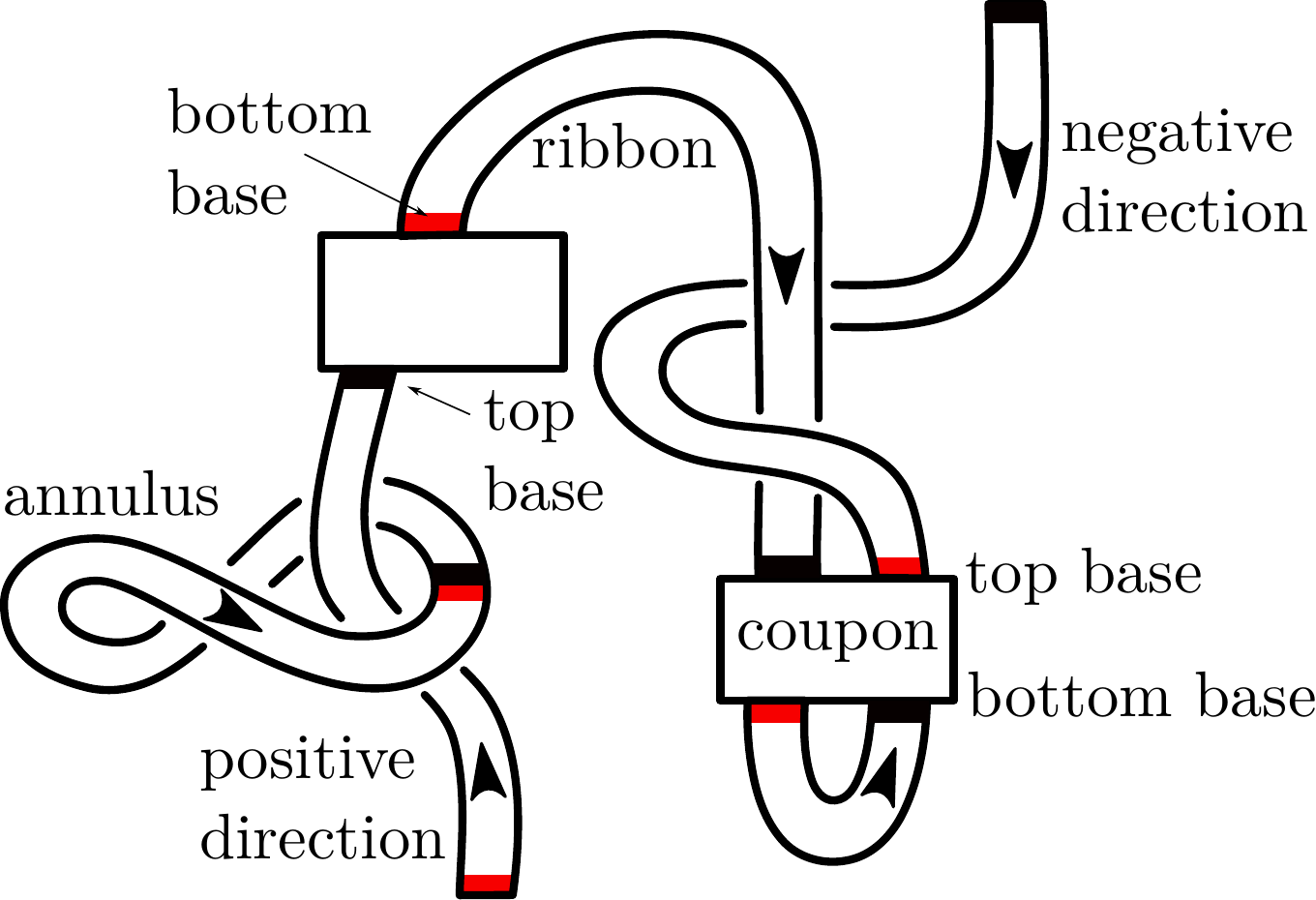}$\qquad$ \includegraphics[height=1.25in]{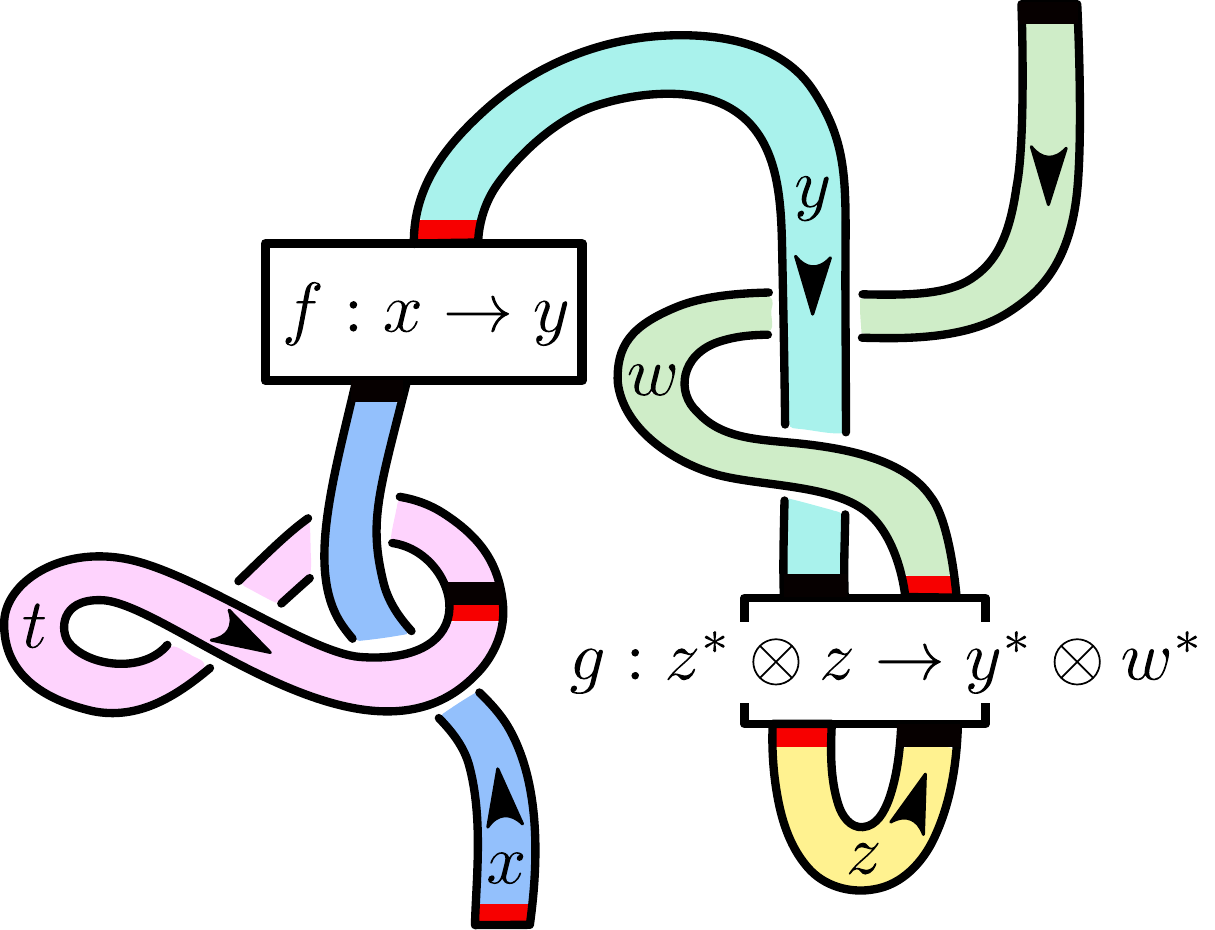}
		\caption{An example of a ribbon graph and its colouring.  Image from \cite[Section 4.2]{Cooke}.}\label{fig:ribbon}
	\end{figure}
	
	Consider the 3-ball $\DD\times I$, and consider a labeling $X\cup Y$ with disks $X=(x_1,V_1),\ldots,(x_n,V_n)$ embedded in $\DD\times\{0\}$ and $Y=\{(y_1,W_1),\ldots (y_m,W_m)\}\times \{1\}$.  Then we have a well-defined surjection, 
	\[\mathrm{Rib}_\cA(\DD\times I,X\cup Y) \to \Hom_\cA(V_1\otimes \cdots \otimes V_n,W_1\otimes \cdots \otimes W_m),\]
	see \cite{TuraevBook}. We will call the kernel of this map the \defterm{skein relations} between $X$ and $Y$.
	
	\begin{definition}
		Let $M$ be an oriented 3-manifold equipped with a decomposition of its boundary $\partial M\cong \rev{\Sigma}_{in}\coprod \Sigma_{out}$, and $\cA$-labelings $X_{in}$ of $\Sigma_{in}$ and $X_{out}$ of $\Sigma_{out}$. \begin{itemize} \item The \defterm{relative $\cA$-skein module} $\SkMod_\cA(M, X_{in}, X_{out})$ is the $\k$-module spanned by isotopy classes of $\cA$-colored ribbon graphs in $M$ compatible with $X_{in}\cup X_{out}$, taken modulo isotopy and the skein relations between $X_{in}$ and $X_{out}$ determined by any oriented ball $\DD\times I\subset M$\footnote{Here we assume without loss of generality that $\DD\times\{0\}\subset \Sigma\times \{0\}$ and $\DD\times\{1\}\subset \Sigma\times \{1\}$.}.
			\item When $\partial M=\emptyset$ (hence $\partial\Gamma=\emptyset$), we call this the \defterm{$\cA$-skein module}, and denote it by $\Sk_\cA(M)$.
		\end{itemize}
	\end{definition}
	
	Using this notion we can define the notion of a skein category of a surface.
	
	\begin{definition}
		Let $\Sigma$ be an oriented surface. The \defterm{skein category} $\SkCat_\cA(\Sigma)$ of $\Sigma$ has:
		\begin{itemize}
			\item As its objects, $\cA$-labelings of $\Sigma$.
			
			\item As the 1-morphisms from $X$ to $Y$ the relative $\cA$-skein module of $(\Sigma\times [0, 1],X,Y)$.
		\end{itemize}
	\end{definition}
	
	The following statement immediately follows from the definitions.
	
	\begin{lemma}
		Let $\rev{\Sigma}$ be the surface with the opposite orientation. Then we have an equivalence
		\[\SkCat_\cA(\rev{\Sigma})\cong \SkCat_\cA(\Sigma)^{\op}\]
		given by sending a labeling $(V_1, \dots, V_n)$ to $(V_1^*, \dots, V_n^*)$ and applying the diffeomorphism $\DD\cong \rev{\DD}$ given by the mirror reflection across the $y$-axis.
		\label{lm:reverseorientation}
	\end{lemma}
	
	The following statement was proved by Walker \cite{Walker}.
	
	\begin{theorem}
		The assignment
		\begin{itemize}
			\item To a closed oriented surface $\Sigma$, the skein category $\SkCat_\cA(\Sigma)$.
			
			\item To an oriented 3-manifold $M$ with a decomposition of its boundary $\partial M\cong \rev{\Sigma}_{in}\coprod \Sigma_{out}$, the functor $\SkMod_\cA(M, -, -)\colon \SkCat_\cA(\Sigma_{in})\times \SkCat_\cA(\Sigma_{out})^{\op}\rightarrow \Vect$ which sends a pair of $\cA$-labelings of $\Sigma_{in}$ and $\Sigma_{out}$ to the relative $\cA$-skein module of $M$.
		\end{itemize}
		defines a 3-dimensional TFT valued in $\Bimod$.
		\label{thm:WalkerTFT}
	\end{theorem}
	
	Note that $\SkCat_\cA(\Sigma)$ has a canonical object $\bu\in\SkCat_\cA(\Sigma)$ given by the empty $\cA$-labeling.
	
	\begin{definition}
		The \defterm{skein algebra of $\Sigma$} is
		\[\SkAlg_\cA(\Sigma) = \End_{\SkCat_\cA(\Sigma)}(\bu).\]
	\end{definition}
	
	\subsection{Examples of skein theories}
	\label{sect:skeinexamples}
	
	In this section we give examples of ribbon categories and their associated skein theories.
	
	Let $G$ be a connected reductive group and fix $q$ not a root of unity. Then $\Repqfd(G)$ is a ribbon category (where the choice of a ribbon structure will be implicit). So, we may consider the \defterm{$G$-skein module}
	\[\Sk_G(M) = \Sk_{\Repqfd(G)}(M)\]
	which is a $k$-vector space for $k=\Q(q^{1/d})$ or $k=\C$. In the cases $G=\SL_2$ and $\SL_3$ the corresponding skein module has a more familiar form as we will explain shortly.
	
	Let us briefly recall the definition of the Temperley--Lieb category, which in \cite{TuraevBook} was called simply the ``skein category'', and which has appeared in many papers since.
	
	\begin{definition}
		For each non-negative integer $m$, fix a finite set $X_m\subset I$ of cardinality $m$.  Given non-negative integers $m$ and $n$, a \defterm{Temperley--Lieb diagram} from $[m]$ to $[n]$ (see e.g. \cref{fig:TL}) is an isotopy class of smoothly embedded compact $1$-manifold $C$ in $I\times I$, such that $\partial C= X_m \times \{0\} \sqcup X_n\times \{1\}$.  Given a Temperley--Lieb diagram $C$, let $u(C)$ denote its number of $S^1$ components, and let $C'$ denote the diagram obtained by omitting the $S^1$ components. 
	\end{definition}
	
	\begin{center}
		\begin{figure}
			\includegraphics[height=1.25in]{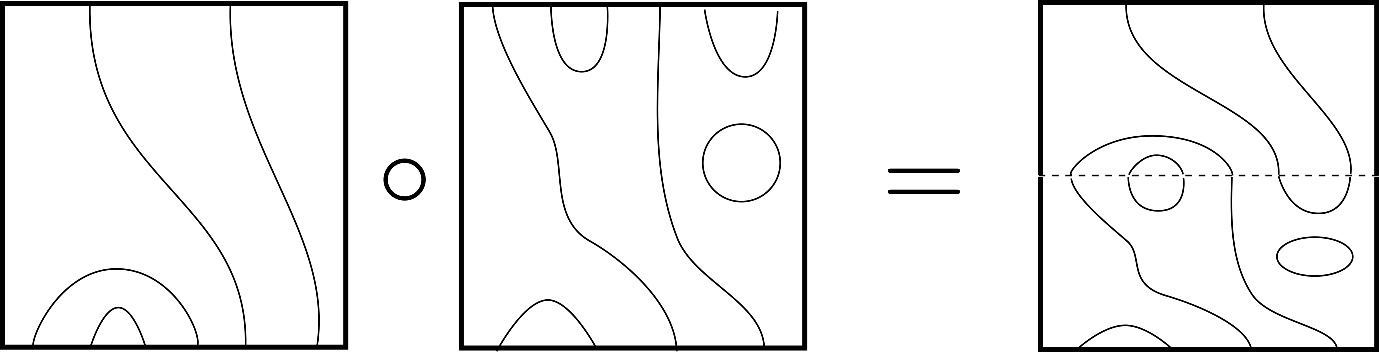}
			\label{fig:TL}
			\caption{The composition of Temperley--Lieb diagrams from $[4]$ to $[6]$ and from $[6]$ to $[2]$, giving a Temperley--Lieb diagram from $[4]$ to $[2]$.}
		\end{figure}
	\end{center}
	
	\begin{definition}[Sketch.  See {\cite[Chapter XII.2]{TuraevBook}}] The \defterm{Temperley--Lieb category} $\TL$ has as objects the non-negative integers $[n]$ and as Hom spaces, the $\Z[A, A^{-1}]$-linear span $\Hom([m],[n])$ of all Temperley--Lieb diagrams, modulo linear relations $C - \delta^{u(C)}C'$, where $\delta=-A^2-A^{-2}$.
		
		Composition of morphisms is given by vertical stacking, and a monoidal structure is given by horizontal stacking; rigidity data is given by the cup and cap diagrams.  A braiding $\sigma$ is defined by setting
		\[\sigma_{[1],[1]} := \KPB := A \KPC + A^{-1} \KPD,\]
		and extending monoidally to all objects $[n]$.  The ribbon element is defined to be $-A^{-3}$ on $[1]$.
	\end{definition}
	
	\begin{remark}
		Let $q=A^2$. For $q$ not a root of unity the Cauchy completion of the Temperley--Lieb category $\TL$ is equivalent to the category $\Repqfd(\SL_2)$. Under this equivalence the object $[1]\in\TL$ goes to the defining two-dimensional representation. The ribbon element on $\Repqfd(\SL_2)$ in this case comes from a half-ribbon element \cite{SnyderTingley} as explained in \cite{Tingley}.
	\end{remark}
	
	\begin{proposition}
		Let $M$ be an oriented 3-manifold. Then we have an isomorphism of $\Z[A, A^{-1}]$-modules
		\[\Sk_{\TL}(M)\cong \Sk(M),\]
		where $\Sk(M)$ is the Kauffman bracket skein module. In particular, for $A$ not a root of unity we have an isomorphism
		\[\Sk_{\SL_2}(M)\cong \Sk(M).\]
		\label{prop:KauffmanTL}
	\end{proposition}
	\begin{proof}
		Let us define a morphism $f\colon\Sk(M)\rightarrow \Sk_{\TL}(M)$ as follows. An element of $\Sk(M)$ is represented by a closed (unoriented) ribbon $s$ in $M$. We assign to $s$ a Temperley--Lieb skein $f(s)$ by choosing an orientation on $s$ and labelling it with the object $[1]\in \TL$. The fact that this does not depend on the choice of orientation corresponds to the statement that the object $[1]\in \TL$ is self-dual with Frobenius--Schur indicator 1. Said equivalently, the defining representation $V$ admits a nondegenerate invariant pairing $\ev\colon V\otimes V\rightarrow \Q$ satisfying
		\[\ev = \ev\circ (\theta\otimes \id) \circ \sigma_{V, V}.\]
		
		The inverse $g$ to $f$ is given as follows: a $\TL$-colored ribbon graph in $M$ consists of a number of ribbons each labelled by some integer $[m]$ and a coupon labelled by a linear combination of Temperley--Lieb diagrams.  For each summand, i.e. for each labelling of each coupon by a single Temperley--Lieb diagram, $g$ assigns a framed link obtained by replacing each ribbon labelled $[m]$ with $m$ parallel strands, and by connecting the incoming and outgoing strands at each coupon using the data of the Temperley--Lieb diagram.  We extend the assignment linearly.
		
		To check that $f$ and $g$ are mutually inverse, it suffices to work locally in any ball in $M$ (as all the relations are local). This amounts to the standard diagrammatics for Temperley--Lieb algebras.
	\end{proof}
	
	It is possible to give a diagrammatic description of $G$-skein modules for other groups analogous to the Kauffman skein relations, though it becomes more complicated.  The first such description was given in \cite{Kuperberg} for $G=\SL_3$; the construction was generalized in \cite{Sikora, CKM} to $G=\SL_N$.  The following presentation is introduced by Kuperberg \cite{Kuperberg}; we follow the description of \cite[Section 1.4]{Sikora}. By a web we mean an oriented ribbon graph whose coupons are either sinks or sources.
	
	\begin{definition}
		Let $M$ be an oriented 3-manifold. The \defterm{Kuperberg skein module} $\Sk_{\SL_3}(M)$ is the $\Z[A, A^{-1}]$-module spanned by trivalent webs in $M$ modulo isotopy and the linear relations,
		\begin{align*}
		\KP{
			\draw[-Latex] (-0.3,-0.3) -- (0.3,0.3);
			\draw[-Latex] (-0.05,0.05) -- (-0.3,0.3);
			\draw (0.05,-0.05) -- (0.3,-0.3);
		}
		&= A^{-1}
		\KP{
			\draw[-Latex] (-0.4, -0.7) -- (-0.1, -0.4);
			\draw[-Latex] (0.4, -0.7) -- (0.1, -0.4);
			\draw (0, -0.3) circle (0.1);
			\draw[-Latex] (0, 0.2) -- (0, -0.2);
			\draw (0, 0.3) circle (0.1);
			\draw[-Latex] (-0.1, 0.4) -- (-0.4, 0.7);
			\draw[-Latex] (0.1, 0.4) -- (0.4, 0.7);
		}
		+ A^2
		\KP{
			\draw[-Latex] (-0.3,-0.3) .. controls (0,0) .. (-0.3,0.3);
			\draw[-Latex] (0.3,-0.3) .. controls (0,0) .. (0.3,0.3);
		} &
		\KP{
			\draw[-Latex] (0.3,-0.3) -- (-0.3,0.3);
			\draw[-Latex] (0.05,0.05) -- (0.3,0.3);
			\draw (-0.05,-0.05) -- (-0.3,-0.3);
		}
		&= A
		\KP{
			\draw[-Latex] (-0.4, -0.7) -- (-0.1, -0.4);
			\draw[-Latex] (0.4, -0.7) -- (0.1, -0.4);
			\draw (0, -0.3) circle (0.1);
			\draw[-Latex] (0, 0.2) -- (0, -0.2);
			\draw (0, 0.3) circle (0.1);
			\draw[-Latex] (-0.1, 0.4) -- (-0.4, 0.7);
			\draw[-Latex] (0.1, 0.4) -- (0.4, 0.7);
		}
		+ A^{-2}
		\KP{
			\draw[-Latex] (-0.3,-0.3) .. controls (0,0) .. (-0.3,0.3);
			\draw[-Latex] (0.3,-0.3) .. controls (0,0) .. (0.3,0.3);
		} &
		\KP {
			\draw (-0.3, -0.3) circle (0.1);
			\draw (0.3, -0.3) circle (0.1);
			\draw (-0.3, 0.3) circle (0.1);
			\draw (0.3, 0.3) circle (0.1);
			\draw[-Latex] (-0.2, 0.3) -- (0.2, 0.3);
			\draw[-Latex] (-0.3, 0.2) -- (-0.3, -0.2);
			\draw[-Latex] (0.2, -0.3) -- (-0.2, -0.3);
			\draw[-Latex] (0.3, -0.2) -- (0.3, 0.2);
			\draw[-Latex] (-0.7, -0.7) -- (-0.37, -0.37);
			\draw[-Latex] (0.37, -0.37) -- (0.7, -0.7);
			\draw[-Latex] (-0.37, 0.37) -- (-0.7, 0.7);
			\draw[-Latex] (0.7, 0.7) -- (0.37, 0.37);
		}
		&=
		\KP {
			\draw[-Latex] (-0.3,-0.3) .. controls (0,0) .. (-0.3,0.3);
			\draw[-Latex] (0.3,0.3) .. controls (0,0) .. (0.3,-0.3);
		}
		+
		\KP {
			\draw[-Latex] (0.3,0.3) .. controls (0,0) .. (-0.3,0.3);
			\draw[-Latex] (-0.3,-0.3) .. controls (0,0) .. (0.3,-0.3);
		}
		\end{align*}
		\begin{align*}
		\KP{
			\draw[-Latex] (0, -0.9) -- (0, -0.5);
			\draw (0, -0.4) circle (0.1);
			\draw[-Latex] (-0.07, 0.33) arc (135:225:0.47);
			\draw[-Latex] (0.07, 0.33) arc (45:-45:0.47);
			\draw (0, 0.4) circle (0.1);
			\draw[-Latex] (0, 0.5) -- (0, 0.9);
		}
		&= (-A^3 - A^{-3})
		\KP {
			\draw[-Latex] (0, -0.8) -- (0, 0.8);
		} &
		\KP {
			\draw (0, 0) circle (0.4);
			\draw[-Latex] (0.4, -0.1) -- (0.4, 0.1);
		}
		&= A^6 + 1 + A^{-6},
		\end{align*}
		which are imposed between any webs agreeing outside of some oriented 3-ball, and differing as depicted inside that ball.
	\end{definition}
	
	Using the results of \cite{Kuperberg} it is straightforward to check that $\Sk_{\SL_3}(M)$ coincides with the skein module for $\Repq(\SL_3)$ equipped with the standard ribbon element, where $q=A^3$.
	
	\subsection{Relation to factorization homology}
	
	Skein categories satisfy a locality property captured by \emph{factorization homology} which will allow us to connect it to the results of \cite{BZBJ1, BZBJ2}.
	
	\begin{definition}
		The bicategory $\Mfld^2$ has:
		\begin{itemize}
			\item As its objects, smooth oriented surfaces,
			\item As the 1-morphisms from $S$ to $T$, all smooth oriented embeddings $S\into T$,
			\item As the 2-morphisms, isotopies of smooth oriented embeddings, themselves considered modulo isotopies of isotopies.
		\end{itemize}
		The disjoint union of surfaces equips $\Mfld^2$ with the structure of a symmetric monoidal bicategory.
	\end{definition}
	
	\begin{definition}
		The bicategory $\Disk^2$ is the full subcategory of $\Mfld^2$ whose objects are finite disjoint unions of oriented disks.
	\end{definition}
	
	The following important and well-known result provides the link between topology and algebra allowing us to compute with factorization homology of braided tensor categories.
	
	\begin{proposition}[{\cite{Fiedorowicz, Dunn, Wahl}}]
		\label{prop:ribbonE2}
		The data of a ribbon tensor category $\cA$ determines a functor $\Disk^2\to\Cat$ which we also denote by $\cA$.
	\end{proposition}
	
	Let us briefly recall the correspondence of data asserted in \cref{prop:ribbonE2}. We denote by $\DD\in\Disk^2$ the standard unit disk with the right-handed orientation.  The tensor product is defined by fixing an embedding $\DD\sqcup \DD\into \DD$, the left-to-right embedding of a pair of smaller disks along the $x$-axis.  The braiding is defined by the isotopy interchanging the embedded disks by rotating them anti-clockwise around one another.  The ribbon element is determined by the the oriented isotopy on $\DD$ rotating it through a 360 degrees turn.  The content of \cref{prop:ribbonE2} is that these embeddings and isotopies taken together freely generate $\Disk^2$, so that once they are specified -- hence the data of a ribbon braided tensor category is fixed -- then the data of the functor is specified uniquely.
	
	The following notion is studied in \cite{AyalaFrancis}, see also \cite{AyalaFrancisReview, TanakaReview}.
	
	\begin{definition}
		The \defterm{factorization homology} $\int_\Sigma \cA$ is the left Kan extension
		\[
		\xymatrix{
			\Disk^2 \ar[rr]^{\cA} \ar[dr]&& \Cat \\
			& \Mfld^2\ar[ur]_{\Sigma\mapsto\int_\Sigma\cA}
		}.
		\]
	\end{definition}
	
	We may analogously define factorization homology internal to $\PrL$ which we denote by $\int_\Sigma^\PrL$.
	
	\begin{lemma}
		We have an equivalence of categories
		\[\fr{\int_\Sigma \cA}\cong \int_\Sigma^{\PrL} \fr{\cA}.\]
		\label{lm:factorizationcocompletion}
	\end{lemma}
	\begin{proof}
		The claim follows since the functor $\fr{-}\colon \Cat\rightarrow \PrL$ preserves colimits.\daj{One of Juliet's thesis referees complained that this is only true for left Kan extensions, not necessarily for aribitrary bicolimits, in the enriched setting, though I find it hard to imagine how that could be}
	\end{proof}
	
	By construction we have $\SkCat_\cA(\DD) \cong \cA$. Cooke has shown that $\SkCat_\cA(-)$ satisfies excision and thus it coincides with factorization homology.
	
	\begin{theorem}[\cite{Cooke}]
		There is an equivalence of categories
		\[\SkCat_\cA(\Sigma) \cong \int_{\Sigma} \cA.\]
		\label{thm:skeinexcision}
	\end{theorem}
	
	We denote
	\[\rZ_\cA(\Sigma) = \int_{\Sigma}^{\PrL} \fr{\cA}\cong \fr{\SkCat_\cA(\Sigma)},\]
	where the second equivalence is provided by \cref{thm:skeinexcision} and \cref{lm:factorizationcocompletion}. The functor $\rZ_\cA(-)$ was studied extensively in \cite{BZBJ1, BZBJ2}.
	
	\subsection{Internal skein algebras}
	\label{sect:internalskeins}
	Recall from \cref{prop:ribbonE2} that $\DD\in\Mfld^2$ is naturally an algebra object. Let $\Ann\subset \bR^2$ be the annulus obtained by removing the disk of radius $1/2$ from the unit disk, both centered at the origin. It has the following algebraic structures as an object of $\Ann$:
	\begin{itemize}
		\item An algebra structure $\Ann\coprod\Ann\rightarrow \Ann$, where the second annulus is put inside the first one.
		\item An algebra map $\DD\rightarrow \Ann$ given by including the disk on the negative $x$-axis.
		\item A map $\Ann\coprod \DD\rightarrow \DD$ given by inserting the disk at the origin which gives $\DD$ a left $\Ann$-module structure.
	\end{itemize}
	
	Suppose $\Sigma\in\Mfld^2$ is a connected oriented surface with a chosen embedding $x\colon \DDc\hookrightarrow \Sigma$. Denote
	\[\Sigma^* = \Sigma\setminus \DDc,\]
	the surface obtained by removing the disk. We have an embedding
	\[\Sigma^*\coprod \Ann\hookrightarrow \Sigma^*\]
	given by retracting away the boundary and including in a copy of the annulus, see \cref{fig:annulus}. This gives $\Sigma^*\in\Mfld^2$ the structure of a right $\Ann$-module. In particular, it is a right $\DD$-module via the algebra map $\DD\rightarrow \Ann$ defined above.
	\begin{figure}
		\includegraphics[height=1.5in]{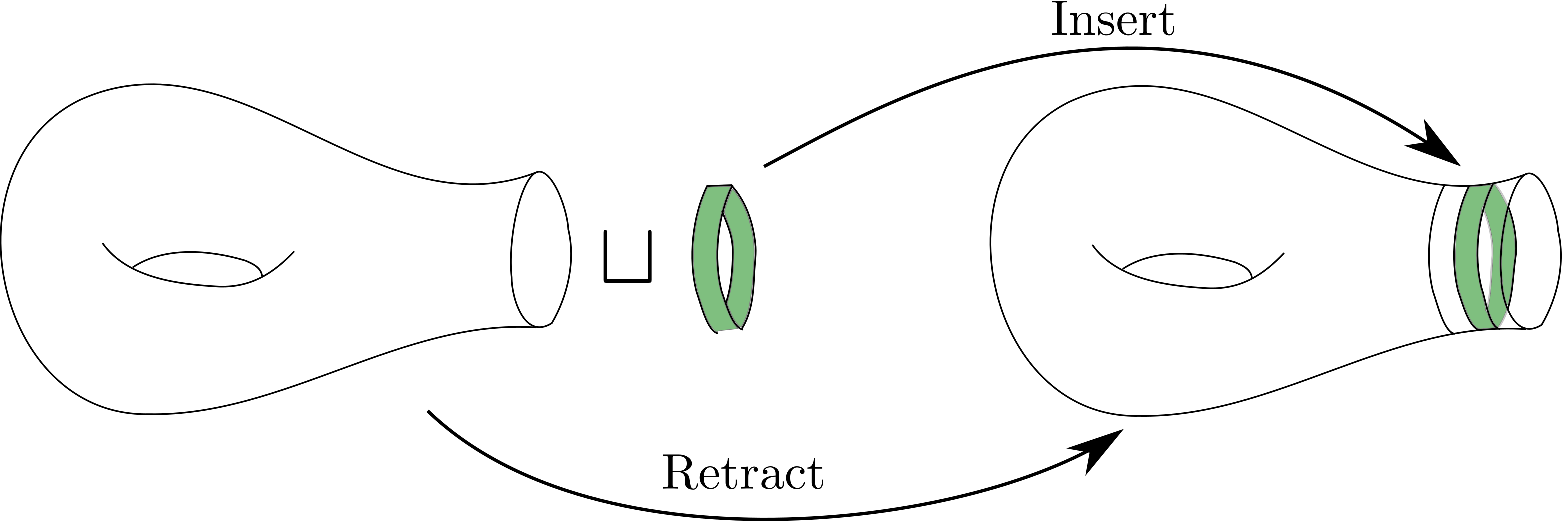}
		\caption{The right $\Ann$-module structure on $\Sigma^*\in\Mfld^2$ comes from boundary insertions.}\label{fig:annulus}
	\end{figure}
	
	On the level of skein categories we obtain a right $\cA$-module category structure on $\SkCat(\Sigma^*)$. Let
	\[\cP\colon \cA \longrightarrow \SkCat_\cA(\Sigma^*)\]
	be the functor given by the action of $\cA \cong \SkCat_\cA(\DD)$ on $\bu\in \SkCat_\cA(\Sigma^*)$.
	
	Recall (see \cref{lm:Dayalgebra}) that an algebra object in $\fr{\cA}$ is the same as a lax monoidal functor $\cA^{\op}\rightarrow \Vect$.
	
	\begin{definition}\label{def:intskeinalg}
		Let $\Sigma$ be a surface as above. The \defterm{internal skein algebra of $\Sigma^*$} is the functor
		\[\SkAlgint_\cA(\Sigma^*)\colon \cA^{\op}\longrightarrow \Vect\]
		given by $V\mapsto \Hom_{\SkCat_\cA(\Sigma^*)}(\cP(V), \bu)$. It has a lax monoidal structure
		\[\Hom_{\SkCat_\cA(\Sigma^*)}(\cP(V), \bu)\otimes \Hom_{\SkCat_\cA(\Sigma^*)}(\cP(W), \bu)\longrightarrow \Hom_{\SkCat_\cA(\Sigma^*)}(\cP(V\otimes W), \bu)\]
		given by stacking the $W$-labeled skein on top of the $V$-labeled skein, see \cref{fig:skein-stack}.
	\end{definition}

\begin{remark}
Unpacking \cref{def:intskeinalg}, we may write the internal skein algebra as a coend,
\[
\SkAlgint_\cA(\Sigma^*) = \int^{X\in \cA} \Hom_{\SkCat_{\cA}(\Sigma^*)}(\cP(X),\bu)\otimes X.
\]

The formula becomes more explicit if we suppose $\cA$ has a fiber functor. In this case we may regard the objects $X\in\cA$ as vector spaces via the fiber functor. Applying the fiber functor to $\SkAlgint_\cA(\Sigma^*)$ we obtain a vector space consisting of skeins in $\Sigma^*\times I$ which are allowed to end with some color $X$ at the distinguished boundary component of $\Sigma^*\times \{0\}$, and which carry an additional label of a vector $x\in X$.  The co-end relations state that a coupon $f:X\to Y$ near the boundary can be absorbed into the boundary by acting as a linear map $X\to Y$.

\end{remark}
	
	\begin{proposition}
		The internal skein algebra $\SkAlgint_\cA(\Sigma^*)\in\fr{\cA}$ is the algebra of $\fr{\cA}$-internal endomorphisms of the distinguished object $\bu\in\SkCat_\cA(\Sigma^*)$.
		\label{prop:internalskeinsendomorphisms}
	\end{proposition}
	\begin{proof}
		The action of $\cA$ on the distinguished object $\bu\in\SkCat_\cA(\Sigma^*)$ is given by $\cP\colon \cA\rightarrow \SkCat_\cA(\Sigma^*)$. Thus, the internal endomorphism algebra $\underline{\End}(\bu)\in\fr{\cA}$ is the functor $\cA^{\op}\rightarrow \Vect$ given by $V\mapsto \Hom_{\SkCat_\cA(\Sigma^*)}(\cP(V), \bu)$ which is exactly the internal skein algebra of $\Sigma^*$.
	\end{proof}
	
	\begin{remark}
		We use the term ``internal skein algebra'' to indicate that $\SkAlgint(\Sigma^*)$ is an algebra internal to the monoidal category $\fr{\cA}$. By \cref{prop:internalskeinsendomorphisms} it is isomorphic to the \emph{moduli algebra} $A_{\Sigma^*}$ from \cite[Definition 5.3]{BZBJ1}.
	\end{remark}
	
	\begin{remark}
		Suppose $\cA = \TL$ is the Temperley--Lieb category. Let $F\colon \TL\rightarrow \Vect$ be the monoidal functor given by the composite $\TL\rightarrow \Repq(\SL_2)\rightarrow \Vect$, where at the end we apply the obvious forgetful functor. We denote by the same letter $F\colon \fr{\TL}\rightarrow \Vect$ the unique colimit-preserving extension. We may write tautologically
		\[\SkAlgint_{\TL}(\Sigma^*)\cong \int^{[n]\in \TL}\Hom_{\SkCat_\TL(\Sigma^*)}(\cP([n]), \bu)\otimes [n]\in\fr{\TL}.\]
		In particular, its underlying vector space is
		\[F(\SkAlgint_{\TL}(\Sigma^*))\cong \int^{[n]\in \TL}\Hom_{\SkCat_\TL(\Sigma^*)}(\cP([n]), \bu)\otimes F([1])^{\otimes n}\in\Vect.\]
		We see that this is exactly the stated skein algebra introduced in \cite{LeTriangular, ConstantinoLe} (see also a related definition of relative skein algebras of \cite{Lofaro}). Namely, $\Hom_{\SkCat_\TL(\Sigma^*)}(\cP([n]), \bu)$ is the vector space of skeins in $\Sigma^*\times [0, 1]$ which have $n$ endpoints on the boundary disk $\DD\hookrightarrow \Sigma^*$ in $\Sigma^*\times [0, 1]$; each endpoint is labeled by a vector in $F([1])$ which is two-dimensional.
	\end{remark}
	
	\begin{remark}
		In particular, the above discussion produces an isomorphism between the stated Kauffman skein algebra of a punctured surface, and the Alekseev-Grosse-Schomerus moduli algebra.  After the present paper first appeared on the arXiv, Matthieu Faitg independently constructed such an isomorphism in \cite{faitg}.  Shortly after that, Benjamin Ha\"ioun gave in \cite{HaiounRelatingStatedSkeinAlgebras} an explicit description of the isomorphism between the internal skein algebra and the stated skein algebra, as asserted in the preceding remark.
	\end{remark}
	
	\begin{figure}
		\includegraphics[height=1.5in]{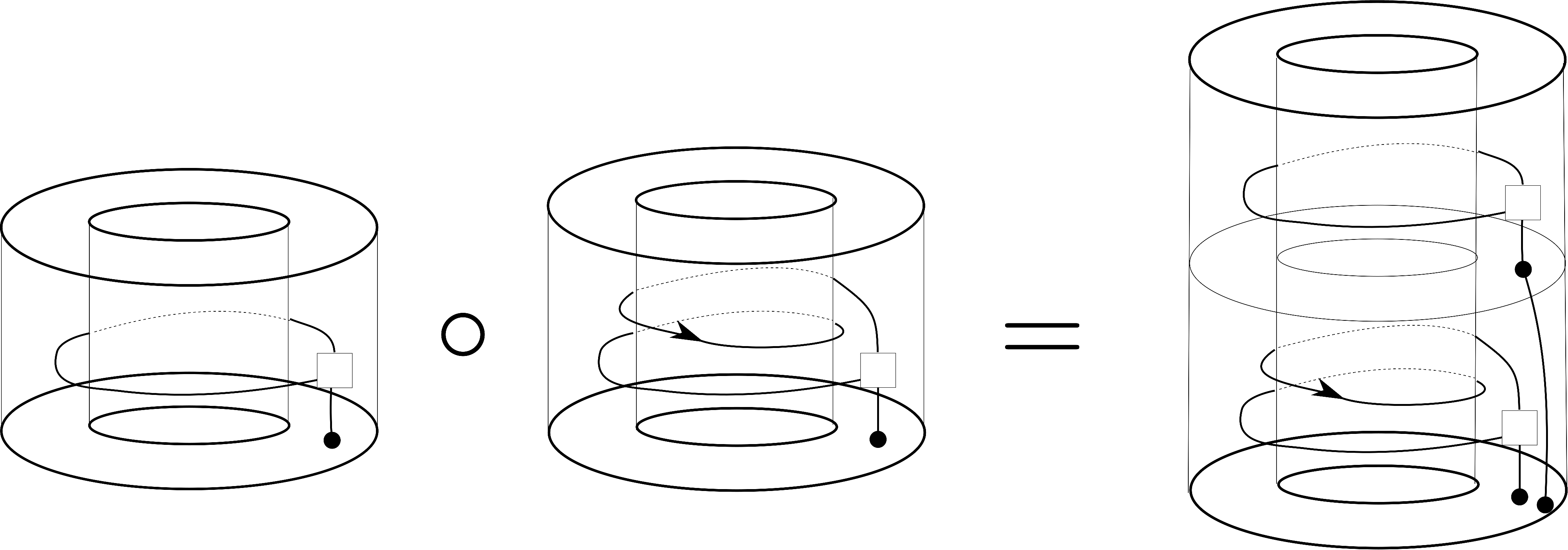}
		\caption{The stacking of internal skeins defines an algebra structure}\label{fig:skein-stack}
	\end{figure}
	
	\begin{remark}
		The skein algebra $\SkAlg_\cA(\Sigma^*)$ of $\Sigma^*$ is the value of $\SkAlgint_\cA(\Sigma^*)$ on $\bu\in\cA$, in other words its $\bu$-multiplicity space, or subalgebra of invariants.
	\end{remark}
	
	Let us now relate internal skein algebras to skein categories. We have a functor
	\[\SkCat_\cA(\Sigma^*)\longrightarrow \fr{\cA}\]
	given by $X\mapsto \Hom_{\SkCat_\cA(\Sigma^*)}(\cP(-), X)$. As for internal skein algebras, we have a stacking morphism
	\[\Hom_{\SkCat_\cA(\Sigma^*)}(\cP(V), \bu)\otimes \Hom_{\SkCat_\cA(\Sigma^*)}(\cP(W), X)\longrightarrow \Hom_{\SkCat_\cA(\Sigma^*)}(\cP(V\otimes W), X).\]
	In other words, we obtain a functor
	\[\SkCat_\cA(\Sigma^*)\longrightarrow \lmod{\SkAlgint_\cA(\Sigma^*)}{\fr{\cA}}.\]
	
	The following statement follows from \cite[Theorem 5.14]{BZBJ1}.
	
	\begin{proposition}
		The functor
		\[\SkCat_\cA(\Sigma^*)\longrightarrow \lmod{\SkAlgint_\cA(\Sigma^*)}{\fr{\cA}}\]
		induces an equivalence
		\[\rZ_\cA(\Sigma^*)=\fr{\SkCat_\cA(\Sigma^*)}\cong\lmod{\SkAlgint_\cA(\Sigma^*)}{\fr{\cA}}.\]
		\label{prop:SkCatmonadicity}
	\end{proposition}
	
	Suppose now $N$ is a compact oriented 3-manifold with $\partial N\cong \Sigma$. The relative skein module defines a functor $\SkMod_\cA(N, -)\colon \SkCat_\cA(\Sigma)^{\op}\rightarrow \Vect$ which we can restrict to a functor $\SkCat_\cA(\Sigma^*)^{\op}\rightarrow \Vect$. Using the equivalence $\fr{\SkCat_\cA(\Sigma^*)}\cong\lmod{\SkAlgint_\cA(\Sigma^*)}{\fr{\cA}}$ given by \cref{prop:SkCatmonadicity} we thus obtain a $\SkAlgint_\cA(\Sigma^*)$-module. Let us describe it explicitly.
	
	\begin{definition}
		Let $N$ be a 3-manifold as above. The \defterm{internal skein module of $N$} is the functor
		\[\Skint_\cA(N)\colon \cA^{\op}\longrightarrow \Vect\]
		given by sending $V\mapsto \SkMod_\cA(N, \cP(V))$. It is a left $\SkAlgint_\cA(\Sigma^*)$-module via the map
		\[\Hom_{\SkCat_\cA(\Sigma^*)}(\cP(V), \bu)\otimes \SkMod_\cA(N, \cP(W))\longrightarrow \SkMod_\cA(N, \cP(V\otimes W))\]
		given by composing the skeins in $\Sigma^*\times [0, 1]$ with skeins in $N$.
		\label{def:intskeinmod}
	\end{definition}
	
	In other words, the internal skein module is given by considering skeins in $N$ which allow to end on $\DD\subset \Ann\subset \Sigma^*\subset \Sigma\cong\partial N$ with label $V\in\cA$. In particular, the ordinary skein module is recovered as
	\[\Sk_\cA(N)\cong \Skint_\cA(N)(\bu).\]
	
	In a similar way, if $N$ is a 3-manifold with $\partial N\cong \rev{\Sigma}$, using \cref{lm:reverseorientation} we define the internal skein module of $N$ to be
	\[V\mapsto \SkMod_\cA(N, \cP(V^*))\]
	which is a \emph{right} $\SkAlgint_\cA(\Sigma^*)$-module.
	
	\subsection{Skein category of the annulus}
	
	We have defined the annulus as $\Ann = \DD \setminus \DDc$, so it makes sense to consider its internal skein algebra. Consider a pair of representations $V, W\in\cA$ together with a morphism $f\colon V\otimes W\rightarrow \bu$. We obtain a skein
	\[s_{V, W, f}\in \Hom_{\SkCat_\cA(\Ann)}(\cP(V\otimes W), \bu)\]
	given by going once around the hole and applying $f$, see \cref{fig:REA}.
	
	\begin{figure}
		\includegraphics[height=1in]{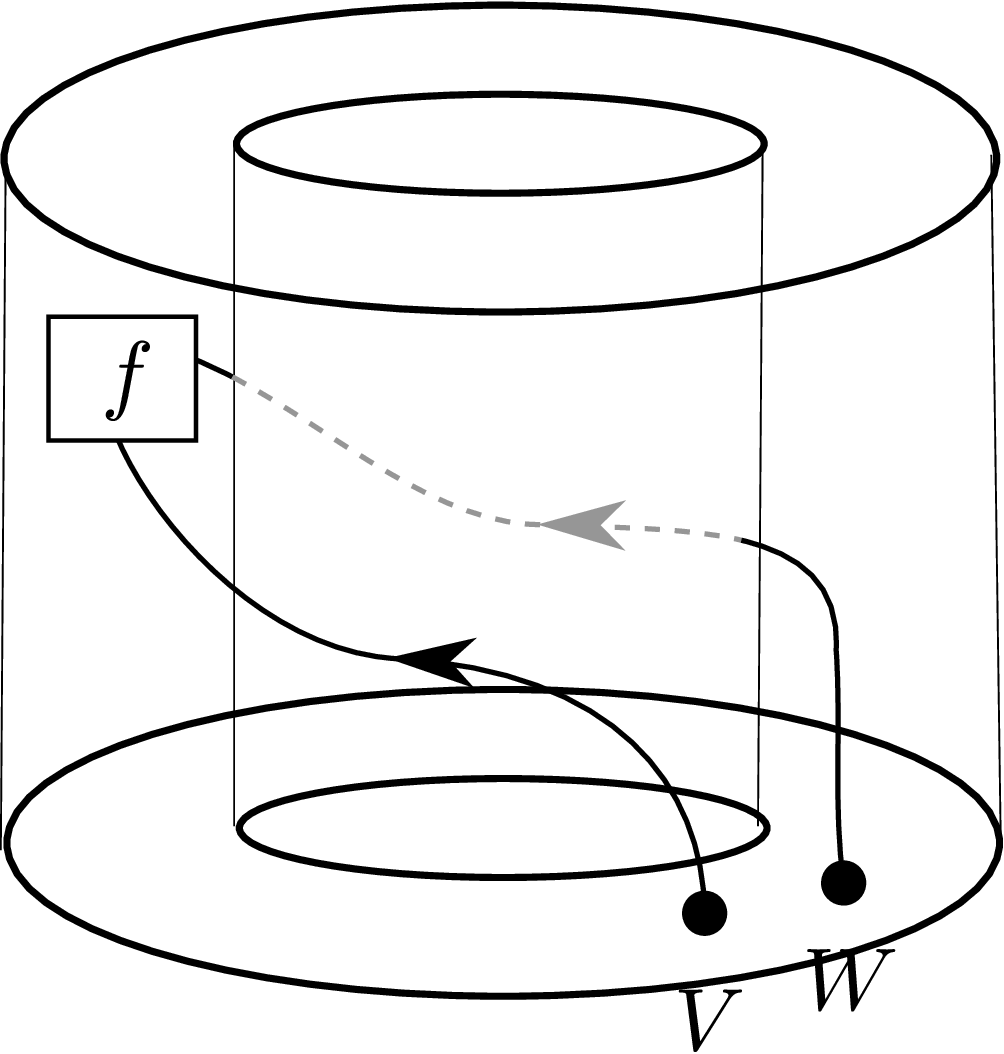}
		\caption{The internal skein $s_{V,W,f}$}\label{fig:REA}
	\end{figure}
	
	For another object $X\in\cA$ we have a composition map
	\[\Hom_\cA(X, V\otimes W)\otimes \Hom_{\SkCat_\cA(\Ann)}(\cP(V\otimes W), \bu)\longrightarrow\Hom_{\SkCat_\cA(\Ann)}(\cP(X), \bu).\]
	
	Thus, applying it to the skein $s_{V, W, f}$ we obtain a map
	\[\Hom_\cA(X, V\otimes W)\longrightarrow \Hom_{\SkCat_\cA(\Ann)}(\cP(X), \bu).\]
	It is natural in $V$ and $W$, so we obtain a morphism
	\[\cF = \left(\underset{V,W\in\cA, f\colon V\otimes W\rightarrow \bu}{\colim} V\otimes W\right)\longrightarrow \SkAlgint_\cA(\Ann).\]
	It is easy to see that it is in fact a morphism of algebras. The following follows from \cite[Corollary 6.4]{BZBJ1}.
	
	\begin{proposition}
		The map $\cF\rightarrow \SkAlgint_\cA(\Ann)$ is an isomorphism.
		\label{prop:annulusREA}
	\end{proposition}
	
	Combining \cref{prop:annulusREA} and \cref{prop:SkCatmonadicity}, we obtain the following statement.
	
	\begin{corollary}
		We have a natural equivalence of categories
		\[\rZ_\cA(\Ann) = \fr{\SkCat}_\cA(\Ann)\cong \HC(\cA).\]
		\label{cor:annulusHC}
	\end{corollary}
	
	We leave it to the reader to check that the monoidal structure on $\HC(\cA)$, the monoidal functor $\cA\rightarrow \HC(\cA)$ and the left $\HC(\cA)$-module structure on $\fr{\cA}$ defined in \cref{sect:HC} go under the above equivalence to the corresponding algebraic structures defined on skein categories in \cref{sect:internalskeins}.
	
	Using the above description of the annulus skein category, we can compute the skein category of a closed surface. Suppose, as before, that $\Sigma$ is a surface with a chosen disk embedding $\DDc\hookrightarrow \Sigma$ and $\Sigma^* = \Sigma \setminus \DDc$. As we have observed in \cref{sect:internalskeins}, $\SkCat_\cA(\Sigma^*)$ is naturally a right $\SkCat_\cA(\Ann)$-module. Therefore, combining \cref{cor:annulusHC} and \cref{prop:HCmomentmap} we obtain a quantum moment map
	\begin{equation}
	\mu\colon \cF\longrightarrow \SkAlgint_\cA(\Sigma^*). \label{eq:skeinalgebraqmm}
	\end{equation}
	In particular, it makes sense to talk about strongly equivariant $\SkAlgint_\cA(\Sigma^*)$-modules.
	
	\begin{proposition}
		We have a natural equivalence of categories
		\[\rZ_\cA(\Sigma) = \fr{\SkCat}_\cA(\Sigma)\cong \lmod{\SkAlgint_\cA(\Sigma^*)}{\fr{\cA}}^{\str}.\]
		\label{prop:closedskeincategory}
	\end{proposition}
	\begin{proof}
		We have a decomposition $\Sigma = \Sigma^*\cup_{\Ann} \DD$. Therefore, by \cref{thm:skeinexcision} we have an equivalence of categories
		\[\SkCat_\cA(\Sigma)\cong \SkCat_\cA(\Sigma^*)\rt{\SkCat_\cA(\Ann)} \cA.\]
		
		Passing to free cocompletions and using \cref{cor:annulusHC} we obtain an equivalence
		\[\rZ_\cA(\Sigma)\cong \rZ_\cA(\Sigma^*)\rt{\HC(\cA)} \fr{\cA}.\]
		
		From \cref{prop:SkCatmonadicity} we get an equivalence
		\[\rZ_\cA(\Sigma)\cong \lmod{\SkAlgint_\cA(\Sigma^*)}{\fr{\cA}}\rt{\HC(\cA)} \fr{\cA}.\]
		
		The claim then follows from \cref{prop:HCstrongequivariance}.
	\end{proof}
	
	\subsection{Skein algebras of surfaces}\label{sec:skeinalgebrasofsurfaces}
	
	Let $\Sigma$ be a closed oriented surface of genus $g$ and let $\Sigma^* = \Sigma \setminus \DDc$ denote the surface obtained by removing some disk in $\Sigma$. Then $\Sigma^*$ has a ``handle and comb'' presentation with $2g$ handles, see \cref{fig:handlebody}. Each handle determines an embedding $\Ann\hookrightarrow \Sigma^*$ and hence an algebra map $\cF\rightarrow \SkAlgint_\cA(\Sigma^*)$. Thus, we obtain a map
	\[\cF^{\otimes 2g}\longrightarrow \SkAlgint_\cA(\Sigma^*)\]
	of objects in $\fr{\cA}$. The following is shown in \cite[Theorem 5.14]{BZBJ1}.
	
	\begin{figure}
		\includegraphics[height=1.5in]{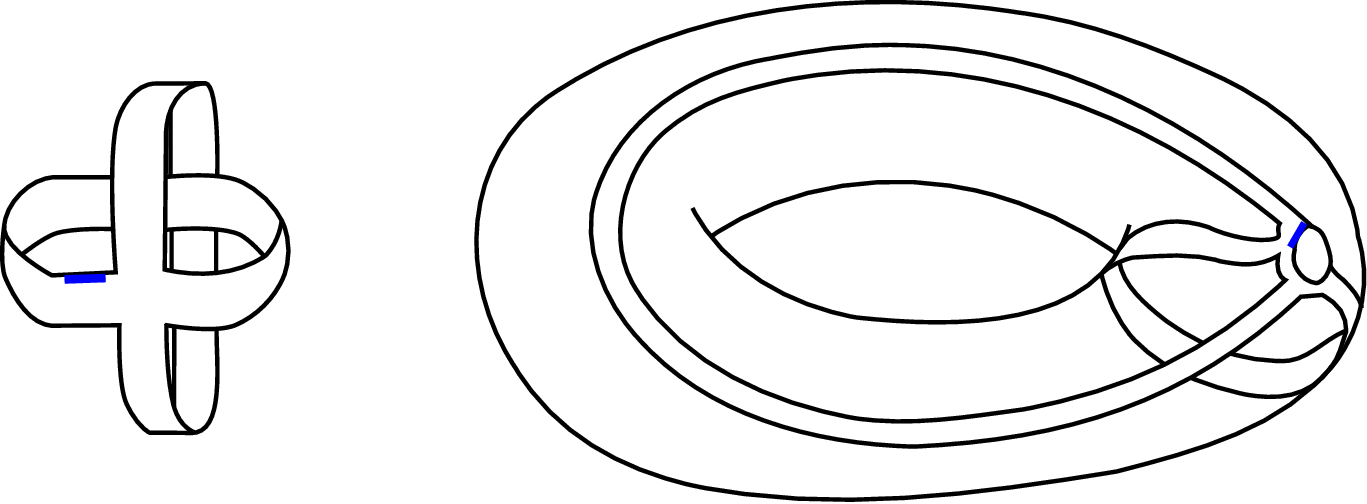}
		\caption{The handle-and-comb decomposition of the once-punctured genus one surface, embedded on the boundary of the genus one handlebody.}\label{fig:handlebody}
	\end{figure}
	
	\begin{proposition}\label{prop:handlecomb}
		The map
		\[\cF^{\otimes 2g}\rightarrow \SkAlgint_\cA(\Sigma^*)\]
		defined above is an isomorphism.
	\end{proposition}
	
	Consider the ring $k=\C\llbracket\hbar\rrbracket$ and let $\cA=\Repqfd(G)$ with $q=\exp(\hbar)$.  Then $\SkAlgint_\cA(\Sigma^*)$ can be considered as an algebra object in vector spaces. The following claim follows from \cite[Section 7.2]{BZBJ1}.
	
	\begin{proposition}
		$\SkAlgint_\cA(\Sigma^*)$ is a flat deformation quantization of $G^{2g}$ with respect to the Fock--Rosly Poisson bracket \cite{FockRosly}.
		\label{prop:fockrosly}
	\end{proposition}
	
	Let
	\[\D_\cA = \SkAlgint_\cA(T^2\setminus\DDc)\]
	be the internal skein algebra in genus 1.

	%
	%\begin{comment}\begin{example}\daj{So if the example stays in the intro, this would come out}	
	%Consider the case $\cA=\Repqfd(\SL_2)$ for generic $q$. Then $\Dq(\SL_2) = \D_\cA$ is known as the Heisenberg (equivalently, elliptic) double of $\Uq(\mathfrak{sl}_2)$ (see \cite{BrochierJordan}). It can be presented with generators $a^1_1, a^1_2, a^2_1, a^2_2, \partial^1_1, \partial^1_2, \partial^2_1, \partial^2_2$ and relations
	%\begin{align*}
	%R_{21}A_1RA_2 &= A_2R_{21}A_1R \\
	%R_{21}D_1RD_2 &= D_2R_{21}D_1R \\
	%R_{21}D_1RA_2 &= A_2R_{21}D_1R_{21}^{-1} \\
	%a^1_1a^2_2 - q^2 a^1_2 a^2_1 &= 1 \\
	%\partial^1_1 \partial^2_2 - q^2 \partial^1_2 \partial^2_1 &= 1.
	%\end{align*}
	%The first three equations take place in $\Dq(\SL_2)\otimes \End(\C^2\otimes \C^2)$, where
	%\begin{align*}
	%A &= \left(\begin{array}{cc} a^1_1 & a^1_2 \\ a^2_1 & a^2_2 \end{array}\right),\qquad A_1 = A\otimes \id,\qquad A_2 = \id\otimes A \\
	%D &= \left(\begin{array}{cc} \partial^1_1 & \partial^1_2 \\ \partial^2_1 & \partial^2_2 \end{array}\right),\qquad D_1 = D\otimes \id,\qquad D_2 = \id\otimes D.
	%\end{align*}
	%
	%It is a flat deformation quantization of $G\times G$ with respect to the Heisenberg double Poisson structure, see \cite{STS}.
	%\label{ex:DqG}
	%\end{example}\end{comment}
	
	Writing a genus $g$ surface as a connected sum of tori, we get an isomorphism of algebras
	\[\D_\cA^{\otimes g}\cong \SkAlgint_\cA(\Sigma^*),\]
	where on the left we consider the braided tensor product of the algebras $\D_\cA$.
	
	\subsection{Handlebody modules}
	\label{sect:handlebodies}
	
	Consider an embedding $\Sigma\hookrightarrow \bR^3$ and let $H$ be its interior. So, $H$ is a handlebody with $\partial H\cong \Sigma$. In particular, it defines a relative skein module
	\[
	\SkMod_\cA(H)\colon \SkCat_\cA(\Sigma)^{\op}\longrightarrow \Vect.
	\]
	As usual, we choose an embedded disk on $\Sigma$ and set $\Sigma^* = \Sigma \setminus \DDc$. As explained in \cref{def:intskeinmod} we can restrict $\SkMod_\cA(H)$ to $\SkCat_\cA(\Sigma^\ast)^{op}$ and obtain a module $\Skint(H)$ for the internal skein algebra $ \SkAlgint_\cA(\Sigma^\ast)$ in $\fr\cA$. In \cref{sec:skeinalgebrasofsurfaces} we gave an explicit description of $\SkAlgint(\Sigma^\ast)$. The goal of this section is to compute the module $\Skint(H)$ in terms of this description. 
	
	Recall the handle and comb presentation of $\Sigma^\ast$ from \cref{sec:skeinalgebrasofsurfaces}. Such a presentation determines a geometric symplectic basis (a system of $a$ and $b$ cycles in $\Sigma^\ast$), i.e., $2g$ embeddings 
	\[
	a_1,b_1, \ldots , a_g, b_g\colon \Ann \to \Sigma^\ast
	\] 
	such that for all $i=1, \ldots, n$ the images of the $a_i$ (respectively $b_i$) are pairwise disjoint, and the intersection of $a_i$ and $b_i$ is a single disk.
	
	Moreover, we choose this system compatible with $H$ in the sense that the $b$-cycles are contractible in $H$. More precisely, we require that each embedding $b_i$ extends to a disk in $H$:
	\[
	\xymatrix{
		\Ann \ar[r]^{b_i} \ar[d] & \Sigma^\ast \ar[d] \\
		\DD \ar[r]& H 
	}
	\]
	
	Let $\Theta$ denote a disk with $g$ smaller disks removed from its interior. The $a$-cycles and the $b$-cycles can be combined to form two embeddings
	\[
	a,b\colon \Theta \hookrightarrow \Sigma
	\]
	See \cref{fig:Theta}.
	
	The following properties are immediate from the construction:
	\begin{itemize}
		\item The handlebody $H$ deformation retracts onto a copy of $\closure{a(\Theta)} \times I$. Indeed, one may begin by considering the manifold with corners $\closure{\Theta}\times I$ then define $\Sigma$ to be some smoothing of its boundary.
		
		\item The embedding $b$ extends over a disk in $H$:
		\[
		\xymatrix{
			\Theta \ar[r]^b \ar[d] & \Sigma^\ast \ar[d] \\
			\DD \ar[r]& H 
		}
		\]
	\end{itemize}
	
	\begin{figure}
		\includegraphics[height=1.5in]{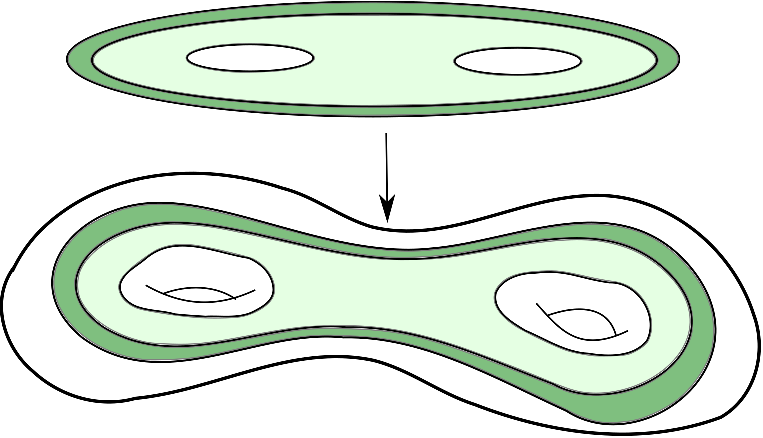}
		\caption{The surface $\Theta$, the embedding $a\colon \Theta\hookrightarrow \Sigma$ and the right $\Ann$-module structure.}\label{fig:Theta}
	\end{figure}
	
	With this set-up in hand, we may now proceed with our computation of the handlebody module. Note that $\Theta$ naturally carries the structure of a right $\DD$-module, by inserting disks inside the ``outer'' annulus in $\Theta$. We can choose the embeddings $a$ and $b$ to be compatible with the right $\DD$-module structure on $\Theta$ and $\Sigma^\ast$.
	
	We obtain the following maps on internal skeins:
	\begin{itemize}
		\item The embeddings $a,b\colon \Theta \hookrightarrow \Sigma^\ast$ determine maps of internal skein algebras
		\[
		i(a),i(b)\colon \SkAlgint_\cA(\Theta) \longrightarrow \SkAlgint_\cA(\Sigma^\ast)
		\]
		
		\item The embedding of $\Theta \hookrightarrow \DD$ determines a map of algebras
		\[
		\varepsilon\colon \SkAlgint_\cA(\Theta) \longrightarrow \SkAlgint_\cA(\DD) = \bu
		\]
		
		\item The composite $\Sigma^\ast \times I \hookrightarrow \Sigma \times I \hookrightarrow (\Sigma \times I) \sqcup_{\Sigma \times \{1\}} H \cong H$ determines a map of left $\SkAlgint_\cA(\Theta)$-modules
		\[\SkAlgint_\cA(\Sigma^\ast) \longrightarrow \Skint_\cA(H).\]
	\end{itemize}
	
	The main result of this section is the following:
	\begin{theorem}\label{thm:handlebody}
		There is an isomorphism of left $\SkAlgint_\cA(\Sigma^\ast)$-modules in $\fr\cA$:
		\[
		\Skint_\cA(H) \cong \SkAlgint_\cA(\Sigma^\ast) \rt{b,\SkAlgint_\cA(\Theta),\varepsilon} \bu
		\]
	\end{theorem}
	
	The proof of this theorem will occupy the rest of this section.
	
	Since the embedding $\Theta\times I\hookrightarrow H$ is a deformation retract, the skein theory of the handlebody can be understood in terms of the internal skein algebra of $a(\Theta)$.
	
	\begin{lemma}\label{lem:a-cyclesisoinH}
		The composite 
		\[
		\SkAlgint_\cA(\Theta) \xrightarrow{a} \SkAlgint_\cA(\Sigma^\ast) \to \Skint(H)
		\]
		is an isomorphism of left $\SkAlgint_\cA(\Theta)$-modules in $\fr \cA$.
	\end{lemma}
	
	In particular, the map $\SkAlgint_\cA(\Sigma^\ast) \to \Skint_\cA(H)$ is surjective. In other words:
	\begin{lemma}\label{lem:handlebodycyclic}
		The $\SkAlgint_\cA(\Sigma^\ast)$ module $\Skint_\cA(H)$ is cyclic, generated by the empty skein.
	\end{lemma}
	
	The next result uses that the $b$-cycles are contractible in the handlebody.
	\begin{lemma}\label{lem:b-cycleskilledinH}
		The action of $\SkAlgint_\cA(\Theta)$ on the empty skein in $\Skint_\cA(H)$ via the inclusion of $b$-cycles factors through $\varepsilon$. In other words, there is a commutative diagram:
		\[
		\xymatrix{
			\SkAlgint_\cA(\Sigma^\ast) \ar[r] & \Skint_\cA(H)\\
			\SkAlgint_\cA(\Theta) \ar[u]_b \ar[r]_\varepsilon & \bu \ar[u] 
		}
		\]
		Thus, there is a well-defined morphism of left $\SkAlgint_\cA(\Sigma^\ast)$-modules
		\[
		\SkAlgint_\cA(\Sigma^\ast) \rt{b,\SkAlgint_\cA(\Theta),\varepsilon} \bu \to \Skint_\cA(H)
		\]
	\end{lemma}
	\begin{proof}
		This follows immediately from the fact that the $b$-cycle embedding
		\[
		\Theta \hookrightarrow \Sigma \hookrightarrow H,
		\]
		factors through the inclusion of a disk.
	\end{proof}
	
	It remains to show that the map $f$ is a isomorphism. To this end, recall from \cref{prop:handlecomb} that the inclusion of both $a$ and $b$ cycles determines an isomorphism in $\fr \cA$
	\[
	\SkAlgint_\cA(\Theta) \otimes \SkAlgint_\cA(\Theta) \xrightarrow{a \otimes b} \SkAlgint_\cA(\Sigma^\ast)
	\]
	Note, that this is not a morphism of algebra objects; however, it is naturally a morphism of right $\SkAlgint_\cA(b(\Theta))$-modules. Thus, we obtain the following:
	\begin{lemma}\label{lem:a-cyclesisototensorproduct}
		The composite
		\[
		\SkAlgint_\cA(\Theta) \xrightarrow{a} \SkAlgint_\cA(\Sigma^\ast) \rt{b,\SkAlgint_\cA(\Theta)} \bu
		\]
		is an isomorphism in $\fr \cA$.
	\end{lemma}
	
	\begin{proof}[Proof of \cref{thm:handlebody}]
		By \cref{lem:b-cycleskilledinH} there is an morphism 
		\[
		f:\SkAlgint_\cA(\Sigma^\ast) \rt{b,\SkAlgint_\cA(\Theta),\varepsilon} \bu \to \Skint_\cA(H)
		\]
		The inclusion of $a$-cycles define a commutative diagram:
		\[
		\xymatrix{
			\SkAlgint_\cA(\Theta) \ar[rd] \ar[rr]^\sim && \Skint_\cA(H) \\
			& \SkAlgint_\cA(\Sigma^\ast) \rt{\SkAlgint_\cA(\Theta)} \bu \ar[ru]_f
		}
		\]
		By \cref{lem:a-cyclesisoinH} the horizontal arrow is an isomorphism. By \cref{lem:a-cyclesisototensorproduct}, the lower right pointing arrow is an isomorphism. It follows that the upper right pointing arrow is an isomorphism as required.
	\end{proof}
	
	\section{Analysis}
	\label{sect:analysis}
	
	This section treats the analytic ingredients in our proof -- completions and localizations in the formal parameter $\hbar$, the finite-dimensionality of localized relative tensor products, deformation quantization modules, and the reduction to $D$-modules. 
	
	\subsection{Completions and localizations}
	
	\begin{definition}
		Let $M$ be a $\C\llbracket\hbar\rrbracket$-module. It is \defterm{$\hbar$-complete} if the map $M\rightarrow \widehat{M} = \lim M/\hbar^n M$ is an isomorphism.
	\end{definition}
	
	\begin{definition}
		Let $M$ be a complex of $\C\llbracket\hbar\rrbracket$-modules.  It is \defterm{cohomologically complete} if
		\[\bR\Hom(\C\llpar\hbar\rrpar, M) = 0.\]
	\end{definition}
	
	\begin{remark}
		The definition of cohomological completeness above agrees with the definition in \cite{KashiwaraSchapiraDQ} due to \cite[Proposition 1.6.(b)]{KashiwaraSchapiraDQ}.  Let us note in passing that without additional assumptions on $\hbar$-torsion, cohomological completeness of a module $M$ does not imply $\hbar$-completeness, nor vice versa.
	\end{remark}
	
	Let us collect from \cite{KashiwaraSchapiraDQ} a number of statements to prove that a module is $\hbar$-complete/cohomologically complete.
	
	\begin{proposition}\label{prop:complete}
		Let $A$ be a $\hbar$-complete $\C\llbracket\hbar\rrbracket$-algebra without $\hbar$-torsion such that $A/\hbar$ is Noetherian.  Let $M$ be an $A$-module.
		\begin{enumerate}
			\item $A$ is Noetherian \cite[Theorem 1.2.5.(i)]{KashiwaraSchapiraDQ}, hence an $A$-module $M$ is finitely generated over $A$ if, and only if, it is coherent as an $A$-module.
			\item If $M$ is a finitely generated $A$-module then $M$ is $\hbar$-complete (\cite[Theorem 1.2.5.(iii)]{KashiwaraSchapiraDQ}) and cohomologically complete (\cite[Theorem 1.6.1]{KashiwaraSchapiraDQ}).\label{prop:test}
			\item Assume $M$ has no $\hbar$-torsion and is $\hbar$-complete.   Then $M$ is cohomologically complete (\cite[Corollary 1.5.7]{KashiwaraSchapiraDQ}, noting that condition (b) is vacuous in our case).
			\item Assume $M$ is cohomologically complete, and let $N$ be a finitely generated right $A$-module.  Then the derived tensor product, $N\otimes^\bL_A M$ is cohomologically complete (\cite[Proposition 1.6.5]{KashiwaraSchapiraDQ}).
		\end{enumerate}
	\end{proposition}
	
	We will also make crucial use of the following ``cohomologically complete Nakayama'' theorem, 
	
	\begin{theorem}[{\cite[Theorem 0.2]{PortaShaulYekutieli}, see also \cite[Theorem 1.6.4]{KashiwaraSchapiraDQ}}]\label{thm:cohomNakayama}
		Let $M$ be a cohomologically complete complex of $\C\llbracket\hbar\rrbracket$-modules, such that $\H^i(M)=0$ for $i>0$, and such that $\H^{0}(\C\rt{\C\llbracket\hbar\rrbracket} M)$ is finitely generated.  Then $\H^{0}(M)$ is finitely generated as a $\C\llbracket\hbar\rrbracket$-module.
	\end{theorem}
	
	\subsection{DQ modules}
	
	Let $X$ be a smooth affine Poisson scheme and $L_1, L_2\subset X$ be smooth Lagrangian subschemes.  Here, by a Lagrangian subscheme of a Poisson scheme we will mean a subscheme of an open symplectic leaf which is Lagrangian there.  In addition, fix their deformation quantizations:
	\begin{itemize}
		\item Let $A$ be a $\hbar$-complete $\C\llbracket\hbar\rrbracket$-algebra without $\hbar$-torsion which is a deformation quantization of $\O(X)$.
		
		\item Let $M_1$ be a cyclic left $A$-module without $\hbar$-torsion which is a deformation quantization of $\O(L_1)$.
		
		\item Let $M_2$ be a cyclic right $A$-module without $\hbar$-torsion which is a deformation quantization of $\O(L_2)$.
	\end{itemize}
	
	In our applications, $A$ will be an internal skein algebra $\SkAlgint(\Sigma_g)$ for a surface $\Sigma_g$, while $M_1$ and $M_2$ will denote the internal skein module for the standard handlebody $H_g$, and its twist $H_g^\gamma$ by a mapping class group, as prescribed by a Heegaard splitting of some 3-manifold.  The remainder of the section is devoted to the proof of the following result.
	
	\begin{theorem}
		The localization
		\[(M_2\rt{A} M_1)[\hbar^{-1}]\]
		is a finite-dimensional $\C\llpar\hbar\rrpar$-vector space.
		\label{thm:holonomictensorproduct}
	\end{theorem}
	
	The proof of \cref{thm:holonomictensorproduct} will be modeled on the proof of constructibility of the derived Hom of holonomic DQ modules in the analytic setting, see \cite[Theorem 7.2.3]{KashiwaraSchapiraDQ}.  \emph{A priori} the tools of \cite{KashiwaraSchapiraDQ} apply only to the analogue of \cref{thm:holonomictensorproduct} for analytic DQ modules and their relative tensor products.  We will therefore repeat the outline of their proof in the algebraic context, which uses the deformation to the normal cone of $L_1$ to reduce the question to one about $D$-modules on $L_1$. Specifically, our definition of $\widehat{A}_{L_1}$ below is motivated by the analogous construction in \cite[Section 7.1]{KashiwaraSchapiraDQ}.
	
	Let us begin by choosing an isomorphism of vector spaces $A\cong \O(X)\llbracket\hbar\rrbracket$.  We obtain an associative multiplication on $\O(X)\llbracket\hbar\rrbracket$, which by \cite[Remark 1.7]{BezrukavnikovKaledin} we may assume is given by a power series of bidifferential operators.  Because a differential operator can only reduce the degree of vanishing along $L_1$ by a finite amount, the multiplication extends to the completion $\O(\widehat{X}_{L_1})\supset \O(X)$ along $L_1$. In this way we obtain a deformation quantization $\widehat{A}\supset A$ of $\O(\widehat{X}_{L_1})$.
	
	Recall that the $A$-module $M_1$ is cyclic, i.e. we have a surjection $A\rightarrow M_1$. In particular, the $\O(X)\llbracket\hbar\rrbracket$-module structure on $\O(L_1)\llbracket\hbar\rrbracket$ is also given by a bidifferential operator. Therefore, the $A$-module structure on $M_1$ extends to an $\widehat{A}$-module structure. Define
	\[\widehat{M}_2 = M_2\rt{A} \widehat{A}\]
	which is a finitely generated $\widehat{A}$-module. Then
	\[M_2\rt{A} M_1\cong \widehat{M}_2\rt{\widehat{A}} M_1.\]
	
	Following \cite{KashiwaraSchapiraDQ}[Section 7.1], we define $J\subset \widehat{A}[\hbar^{-1}]$ to be the kernel of
	\[
	\hbar^{-1}\widehat{A}\xrightarrow{\hbar} \widehat{A} \rightarrow \O(\widehat{X}_{L_1})\rightarrow \O(L_1)
	\]
	and denote by $\widehat{A}_{L_1}\subset \widehat{A}[\hbar^{-1}]$ the $\C\llbracket\hbar\rrbracket$-subalgebra generated by $J$.
	
	\begin{proposition}
		The inclusion $\widehat{A}_{L_1}\subset \widehat{A}[\hbar^{-1}]$ induces an equality
		\[\widehat{A}_{L_1}[\hbar^{-1}] = \widehat{A}[\hbar^{-1}].\]
	\end{proposition}
	
	Thus, both $\widehat{A}_{L_1}$ and $\widehat{A}$ are $\C\llbracket\hbar\rrbracket$-lattices in $\widehat{A}[\hbar^{-1}]$. We can therefore reduce questions about $\widehat{A}$-modules to questions about $\widehat{A}_{L_1}$-modules.
	
	\begin{proposition}
		Suppose $N_1$ is a finitely generated left $\widehat{A}_{L_1}$-module without $\hbar$-torsion together with an isomorphism of left $\widehat{A}[\hbar^{-1}]$-modules $N_1[\hbar^{-1}]\cong M_1[\hbar^{-1}]$ and similarly for $N_2$. Suppose $N_2/\hbar\rt{\widehat{A}_{L_1}/\hbar} N_1/\hbar$ is a finite-dimensional $\C$-vector space. Then $(\widehat{M}_2\rt{\widehat{A}} M_1)[\hbar^{-1}]$ is a finite-dimensional $\C\llpar\hbar\rrpar$-vector space.
		\label{prop:fglattice}
	\end{proposition}
	\begin{proof}
		Since $\widehat{A}_{L_1}$ is $\hbar$-complete and without $\hbar$-torsion, Claim (2) of \cref{prop:complete} implies it is cohomologically complete. Since $N_1$ is finitely generated as an $\widehat{A}_{L_1}$-module, Claim (1) of \cref{prop:complete} implies $N_1$ is cohomologically complete. Since $N_2$ is finitely generated, Claim (3) of \cref{prop:complete} implies that the derived tensor product $N_2\rtL{\widehat{A}_{L_1}} N_1$ is cohomologically complete.
		
		Since $M=N_2\rtL{\widehat{A}_{L_1}}\!\! N_1$ is concentrated in non-positive cohomological degrees and since
		\[\H^0(\C \rtL{\C\llbracket\hbar\rrbracket} M) = N_2/\hbar\rt{\widehat{A}_{L_1}/\hbar} N_1/\hbar\]
		is assumed to be finite-dimensional as a $\C$-vector space, the cohomologically complete Nakayama \cref{thm:cohomNakayama} implies that
		\[\H^0(M) = N_2\rt{\widehat{A}_{L_1}} N_1\]
		is finitely generated as a $\C\llbracket\hbar\rrbracket$-module.  In particular,
		\[
		(\widehat{M}_2\rt{\widehat{A}} M_1)[\hbar^{-1}]
		\cong (\widehat{M}_2[\hbar^{-1}])\rt{\widehat{A}[\hbar^{-1}]}(M_1[\hbar^{-1}])
		\cong (N_2[\hbar^{-1}])\rt{\widehat{A}[\hbar^{-1}]}(N_1[\hbar^{-1}])
		\cong (N_2\rt{\widehat{A}_{L_1}} N_1)[\hbar^{-1}].\]
		is a finite-dimensional $\C\llpar\hbar\rrpar$-vector space.
	\end{proof}
	
	\subsection{Reduction to \texorpdfstring{$D$}{D}-modules}
	
	In this section we work near a Lagrangian to reduce questions about DQ modules to questions about ordinary D-modules. First, we will need the following lemma.
	
	\begin{lemma}
	Let $L$ be a smooth scheme and consider the completion $\widehat{\T}^* L$ of its cotangent bundle along the zero section. Denote by $i\colon L\hookrightarrow \widehat{\T}^* L$ the natural embedding. Then the restriction map
	\[i^*\colon \H^\bullet_{dR}(\widehat{\T}^*L)\longrightarrow \H^\bullet_{dR}(L)\]
	on de Rham cohomology is an isomorphism.
	\label{lm:dRrestrictionisomorphism}
	\end{lemma}
	\begin{proof}
	Indeed, the claim follows from \cite[Proposition II.1.1]{Hartshorne} applied to $X=Z=L$ and $Y=\T^* L$.
	\end{proof}
	
	\begin{theorem}[Lagrangian neighborhood theorem]
		Let $X$ be an affine symplectic scheme and $L\subset X$ a smooth Lagrangian subscheme. Then there is a symplectomorphism of formal symplectic schemes $\widehat{X}_L\cong \widehat{\T}^*L$.
	\end{theorem}
	\begin{proof}
		By \cite[Lemma 5.2]{CCT} we may identify $\widehat{X}_L\cong \widehat{\N} L$ as formal schemes. Since $L$ is Lagrangian, we may identify $\widehat{\N} L\cong \widehat{\T}^* L$ as formal schemes. Thus, we obtain two symplectic structures on $\widehat{\T}^* L$: $\omega_0$ coming from the cotangent bundle and $\omega_1$ coming from $\widehat{X}_L$. To prove the claim, we will use Moser's trick.
		
		By \cref{lm:dRrestrictionisomorphism} the restriction map $\H^2_{dR}(\widehat{\T}^* L)\rightarrow \H^2_{dR}(L)$ is an isomorphism. Therefore, $\omega_1 - \omega_0 = d \alpha$. Moreover, since $\omega_1|_L = \omega_0|_L$, we may arrange $\alpha$ so that $\alpha|_L = 0$. Consider the family of closed 2-forms
		\[\omega_t = \omega_0 + t d\alpha.\]
		By assumption $L\hookrightarrow \widehat{\T}^* L$ is isotropic for the whole family. So, to check that $\omega_t$ is symplectic, it is enough to check that $\omega_t^\sharp\colon \T_L\rightarrow \N^*_L$ is an isomorphism. But it immediately follows from the fact that $L$ is Lagrangian with respect to $\omega_t$ and $\alpha|_L = 0$.
		
		Since $\omega_t$ is symplectic, we may find a time-dependent vector field $v_t$ which satisfies Moser's equation $\iota_{v_t}\omega_t = -\alpha$ and which vanishes on $L$. We may integrate this vector field to an isotopy $\rho_t$ which by Moser's equation satisfies $\rho_t^*\omega_t = \const$. In particular, $\rho_1$ is an automorphism of $\widehat{\T}^* L$ preserving $L$ such that $\rho_1^*\omega_1 = \omega_0$.
	\end{proof}
	
	Let $\D(L_1)$ be the $\C$-algebra of differential operators. It admits a filtration given by the order of the differential operator. Consider the Rees algebra which is a graded $\C[\hbar]$-algebra and complete it in the $\hbar$-adic topology as well as with respect to the order filtration. We denote the completion by $\widehat{\D}_\hbar(L_1)$. The algebra $\widehat{\D}_\hbar(L_1)$ is flat over $\C\llbracket\hbar\rrbracket$ and we have an isomorphism
	\[\widehat{\D}_\hbar(L_1)/\hbar\cong \O(\widehat{\T}^*L_1)\]
	of Poisson algebras. In particular, $\widehat{\D}_\hbar(L_1)$ provides a deformation quantization of $\widehat{\T}^* L_1$.
	
	\begin{proposition}
		There is an isomorphism of algebras
		\[\widehat{\D}_\hbar(L_1)\cong \widehat{A}\]
		lifting a symplectomorphism $\widehat{\T}^* L_1\cong \widehat{X}_{L_1}$.
	\end{proposition}
	\begin{proof}
		By \cite[Theorem 1.8]{BezrukavnikovKaledin} (which is valid for formal symplectic varieties as well) deformation quantizations $\widehat{A}$ of $\widehat{\T}^* L_1$ are classified by their periods
		\[\Per(\widehat{A})\in \hbar \H^2_{dR}(\widehat{\T}^* L_1)\llbracket\hbar\rrbracket.\]
		
		If $B$ is a deformation quantization of $\widehat{\T}^*L_1$ which admits a module quantizing $L_1\subset \widehat{\T}^*L_1$, by \cite[Theorem 1.1.4, Lemma 5.3.5]{BGKP} we have
		\[i^*\Per(B)=-\hbar c_1(K_{L_1})/2,\]
		where $c_1(K_{L_1})\in\H^2_{dR}(L_1)$ is the first Chern class of the canonical bundle of $L_1$.
		
		The algebras $\widehat{\D}_\hbar(L_1)$ and $\widehat{A}$ are each deformation quantizations of $\widehat{\T}^*L_1$. The $\widehat{\D}_\hbar(L_1)$-module $\O(L_1)\llbracket\hbar\rrbracket$ and the $\widehat{A}$-module $M_1$ are each deformation quantizations of the Lagrangian $L_1\subset \widehat{\T}^*L_1$. Hence, we have
		\[i^* \Per(\widehat{A}) = -\hbar c_1(K_{L_1})/2 = i^*\Per(\widehat{\D}_\hbar(L_1)).\]
		By \cref{lm:dRrestrictionisomorphism} this implies that $\Per(\widehat{A}) = \Per(\widehat{\D}_\hbar(L_1))$. Therefore, the two deformation quantizations are isomorphic.
	\end{proof}
	
	Unpacking the definitions, the lattice
	\[\widehat{A}_{L_1}\subset \widehat{\D}_\hbar(L_1)[\hbar^{-1}]\]
	is generated by functions $f$ and $\tilde{v}=\hbar^{-1}v$ for vector fields $v$. Therefore, we obtain an isomorphism
	\[\widehat{A}_{L_1} / \hbar\cong \D(L_1).\]
	
	\begin{remark}
		Note that $\widehat{A}_{L_1} / \hbar\cong \D(L_1)$ is noncommutative while $\widehat{A}/\hbar\cong \O(\widehat{\T}^* L_1)$ is commutative. These are two different ways to take the $\hbar\rightarrow 0$ limit of the algebra $\widehat{A}[\hbar^{-1}]$.
	\end{remark}
	
	\begin{proposition}
		In the setup of Section 3.2, there is an $\widehat{A}_{L_1}$-lattice $N_2$ in the $\widehat{A}[\hbar^{-1}]$-module $\widehat{M}_2[\hbar^{-1}]$ such that the $\D(L_1)$-module $N_2/\hbar$ is holonomic.
		\label{prop:holonomiclattice}
	\end{proposition}
	\begin{remark}
		In fact, as mentioned in \cite[Lemma 7.1.12]{KashiwaraSchapiraDQ}, the holonomicity of $N_2/\hbar$ is independent of the choice of lattice. However, we will only need a single such choice.
	\end{remark}
	
	\begin{remark} Let us remark that \cref{prop:holonomiclattice} is essentially a special case of \cite[Proposition 7.1.16]{KashiwaraSchapiraDQ}: the proofs in \emph{loc. cit} are entirely algebraic in nature, and use only standard homological properties of $D$-modules and DQ-modules which hold as well in the algebraic setting as in the analytic.  Because we are only interested in the case when the support of $\widehat{M}_2$ is Lagrangian, the argument can be somewhat simplified, but otherwise applies nearly verbatim in our setting.\end{remark}
	
	We will require the following lemma concerning duality of $\widehat{A}$-modules, which is the analog of \cite[Proposition 2.3.11]{KashiwaraSchapiraDQ} in our setting (see also \cite[Theorem 2.6.6 \& D.4.3]{HottaTakeuchiTanisaki}).  
		%Let $A^{loc} = \widehat{A}[\hbar^{-1}]$ and $\widehat{M}_2[\hbar^{-1}] = M_2^{loc}$.

\begin{lemma}\label{lem:KS}
	The cohomology of the complex of $\widehat{A}$-modules
	\[
	\RHom_{\widehat{A}^{\op}}(\widehat{M}_2, \widehat{A}^{\op})
	\]
	is concentrated in degree $d=\dim(X)/2$.
	Moreover, we have an isomorphism of $\widehat{A}^\op$-modules
	\[
	\widehat{M_2} \cong \Ext^d_{\widehat{A}}(M',\widehat{A}),
	\]
	where $M'$ is the finitely generated $\hat{A}$-module defined by
	\[
	M' = \Ext^d_{\widehat{A}^{\op}}(\widehat{M}_2,\widehat{A}^{\op}).
	\]
\end{lemma}
\begin{proof}
The proof of this statement found in \cite[Proposition 2.3.11]{KashiwaraSchapiraDQ} applies essentially verbatim in our situation, after replacing the sheaf of $\C\llbracket\hbar\rrbracket$-algebras $\cA_X$ with the $\C\llbracket\hbar\rrbracket$-algebra $\widehat{A}$. The key ingredient is that the corresponding statements are true at the classical level (i.e. taking $\hbar=0$), which one can see by taking a Koszul resolution (locally) for $\O(L_2)$ as a $\O(X)$-module, and completing at $L_1$.
\end{proof}

%	\begin{proof}
%	The proofs in \cite[Proposition 2.3.15 \& Proposition 1.4.5]{KashiwaraSchapiraDQ} apply verbatim, replacing the sheaf of $\C\llbracket\hbar\rrbracket$-algebras $\cA_X$ in \emph{loc. cit.} with the $\C\llbracket\hbar\rrbracket$-algebra $\widehat{A}$, and the sheaf of modules $\cM$ with the module $\widehat{M}_2$.  
%	\end{proof}
		%By \cite[Proposition 2.3.15]{KashiwaraSchapiraDQ} its cohomology is concentrated in degree $d=\dim(L_1)$,

	\begin{proof}[Proof of \cref{prop:holonomiclattice}]
	Note that the results of Lemma \ref{lem:KS} remain true after inverting $\hbar$, as localization is exact. Thus we can write 
	\[
	\widehat{M_2}[\hbar^{-1}] \cong \Ext^d_{\widehat{A}[\hbar^{-1}]}(M'[\hbar^{-1}],\widehat{A}[\hbar^{-1}]),
	\]	
		The argument now runs in parallel to the second paragraph of the proof of \cite[Proposition 7.1.16]{KashiwaraSchapiraDQ} (the reduction in the first paragraph is taken care of by Lemma \ref{lem:KS}). 
		
		Choosing an $\widehat{A}_{L_1}$-lattice $N'$ of $M'[\hbar^{-1}]$, we obtain a right $\widehat{A}_{L_1}$-module
		\[
		N'' := \Ext^d_{\widehat{A}_{L_1}}(N', \widehat{A}_{L_1}),
		\]
		such that $N''[\hbar^{-1}]\cong \widehat{M}_2[\hbar^{-1}]$. In particular, we have a morphism $N''\rightarrow \widehat{M}_2[\hbar^{-1}]$ and we define $N_2$ to be the image, which is an $\widehat{A}_{L_1}$-lattice of $\widehat{M}_2[\hbar^{-1}]$.
		
		In particular, we have a surjection $N''/\hbar\rightarrow N_2/\hbar$ of $\D(L_1)$-modules, so the claim follows once we show that $N''/\hbar$ is holonomic (as holonomicity is preserved by quotients). But $N''/\hbar$ naturally embeds as a submodule of $\Ext^d_{\D(L_1)}(N'/\hbar, \D(L_1))$ which is holonomic by \cite[Theorem 2.6.7]{HottaTakeuchiTanisaki}, and thus $N''/\hbar$ is also holonomic as required.
	\end{proof}
	
	We can now finish the proof of the main theorem of the section.
	
	\begin{proof}[Proof of \cref{thm:holonomictensorproduct}]
		By \cref{prop:holonomiclattice} choose $\widehat{A}_{L_1}$-lattices $N_1$ and $N_2$ for $M_1$ and $\widehat{M}_2$ respectively such that $N_1'=N_1/\hbar$ and $N_2/\hbar$ are holonomic $\D(L_1)$-modules. Consider the left $\D(L_1)$-module $N_2' = \RHom_{\D(L_1)^{\op}}(N_2/\hbar, \D(L_1))$ which is also holonomic by \cite[Corollary 2.6.8]{HottaTakeuchiTanisaki}. By \cite[Lemma 2.6.13]{HottaTakeuchiTanisaki} we have
		\[N_2/\hbar\otimes^\bL_{\D(L_1)} N_1/\hbar\cong \RHom_{\D(L_1)}(N_2', N_1').\]
		Applying \cite[Corollary 2.6.15]{HottaTakeuchiTanisaki}, by preservation of holonomicity (\cite[Theorem 3.2.3]{HottaTakeuchiTanisaki}) this is a bounded complex with finite-dimensional cohomology. Therefore, by \cref{prop:fglattice} $(\widehat{M}_2\rt{\widehat{A}} M_1)[\hbar^{-1}]$ is a finite-dimensional $\C\llpar\hbar\rrpar$-vector space.
	\end{proof}
	
	\section{Applications}
	\label{sect:applications}
	
	This section brings together the ingredients from the preceding three sections, to prove our main results.
	
	\subsection{Relative tensor product}
	
	\label{sect:relativetensorproduct}
	
	Our first goal is to prove the tensor product formula for the skein module of a 3-manifold.
	
	Let $\Sigma$ be a connected closed oriented surface and $N_1, N_2$ are oriented 3-manifolds with boundary such that $\partial N_1\cong \Sigma$ and $\partial N_2\cong \rev{\Sigma}$ and let
	\[M = N_2\cup_{\Sigma} N_1.\]
	Choose a disk embedding $\DDc\hookrightarrow \Sigma$ and let $\Sigma^* = \Sigma\setminus \DDc$.  Let $A$ denote $\SkAlgint_\cA(\Sigma^*)$, equipped with its quantum moment map \eqref{eq:skeinalgebraqmm}.
	
	\begin{theorem}
		There is an isomorphism
		\[\Sk_\cA(M)\cong \Hom_{\fr{\cA}}\left(\bu, \Skint_\cA(N_2)\rt{A} \Skint_\cA(N_1)\right).\]
		\label{thm:relativetensorproduct}
	\end{theorem}
	\begin{proof}
		Let
		\[\SkMod_\cA(N_1)\colon \SkCat_\cA(\Sigma)^{\op}\rightarrow \Vect,\qquad \SkMod_\cA(N_2)\colon \SkCat_\cA(\Sigma)\rightarrow \Vect\]
		be the relative skein modules for $N_1$ and $N_2$. By the TFT property (see \cref{thm:WalkerTFT}) we have
		\[\Sk_\cA(M)\cong \SkMod_\cA(N_2)\rt{\SkCat_\cA(\Sigma)} \SkMod_\cA(N_1).\]
		
		By \cref{ex:Homnondegenerate} the relative tensor product defines a nondegenerate pairing between $\fr{\SkCat_\cA(\Sigma)}=\Fun(\SkCat_\cA(\Sigma)^{\op}, \Vect)$ and $\Fun(\SkCat_\cA(\Sigma), \Vect)$. By \cref{prop:stronglyequivariantduality} the functor
		\[X_1, X_2\mapsto \Hom_{\fr{\cA}}(\bu, X_2\rt A X_1)\]
		defines a nondegenerate pairing between $\lmod{A}{\fr{\cA}}^{\str}$ and $\rmod{A}{\fr{\cA}}^{\str}$. In particular, it is enough to restrict all modules from $\Sigma$ to $\Sigma^*$.
		
		By definition the internal skein modules $\Skint_\cA(N_1)$ and $\Skint_\cA(N_2)$ are the images of $\SkMod_\cA(N_1)$ and $\SkMod_\cA(N_2)$ under the functors
		\[\fr{\SkCat_\cA(\Sigma^*)}\rightarrow \lmod{A}{\fr{\cA}},\qquad \fr{\SkCat_\cA(\Sigma^*)^{\op}}\rightarrow \rmod{A}{\fr{\cA}}\]
		which send $\cP(V)\mapsto A\otimes V$ and $\cP(V)\mapsto V^*\otimes A$ respectively.
		
		The claim is reduced to the commutativity of the diagram
		\[
		\xymatrix{
			{\fr{\SkCat_\cA(\Sigma^*)^{\op}}\otimes \fr{\SkCat_\cA(\Sigma^*)}} \ar[dr] \ar[dd] & \\
			& \Vect \\
			{\rmod{A}{\fr{\cA}}\otimes \lmod{A}{\fr{\cA}}} \ar[ur] &
		}
		\]
		
		It is enough to check it on the generating objects $\cP(V), \cP(W)$ for $V,W\in\cA$. Their image under the evaluation pairing on $\fr{\SkCat_\cA(\Sigma^*)}$ is
		\begin{align*}
		\Hom_{\SkCat_\cA(\Sigma^*)}(\cP(V), \cP(W))&\cong \Hom_{\fr{\cA}}(V, \cP^\R\cP(W)) \\
		&\cong \Hom_{\fr{\cA}}(V, A\otimes W).
		\end{align*}
		
		Similarly, their image under the evaluation pairing on $\lmod{A}{\fr{\cA}}$ is
		\[\Hom_{\fr{\cA}}(\bu, (V^*\otimes A)\rt{A} (A\otimes W))\cong \Hom_{\fr{\cA}}(\bu, V^*\otimes A\otimes W)\]
		which is equivalent to the previous pairing using rigidity of $\cA$.
	\end{proof}
	
	In the case when $\fr{\cA}$ has a trivial M\"uger center, the claim simplifies.
	
	\begin{corollary}
		Suppose $\fr{\cA}$ has a trivial M\"uger center. Then there is an isomorphism
		\[\Sk_\cA(M)\cong \Skint_\cA(N_2)\rt{A} \Skint_\cA(N_1)\in \ZMug(\fr{\cA})\cong\Vect.\]
		\label{cor:relativetensorproductMuger}
	\end{corollary}
	\begin{proof}
		Indeed, by \cref{prop:stronglyequivarianttensorproductMuger}
		\[\Skint_\cA(N_2)\rt{A} \Skint_\cA(N_1)\in \ZMug(\fr{\cA})\cong\Vect\]
		since both internal skein modules are strongly equivariant. The claim then follows follows from \cref{thm:relativetensorproduct} since the unit object of $\cA$ is simple.
	\end{proof}
	
	In the case of a Heegaard splitting, the relative tensor product formula simplifies.
	
	\begin{proposition}
		Suppose $N_1, N_2$ are handlebodies. Then there is an isomorphism
		\[\Sk_\cA(M)\cong \Sk_\cA(N_2)\rt{\SkAlg_\cA(\Sigma)} \Sk_\cA(N_1).\]
		\label{prop:tensorproducthandlebody}
	\end{proposition}
	\begin{proof}
		As before, by the TFT property (\cref{thm:WalkerTFT}) we have
		\[\Sk_\cA(M)\cong \SkMod_\cA(N_2)\rt{\SkCat_\cA(\Sigma)} \SkMod_\cA(N_1).\]
		
		By \cref{lem:handlebodycyclic} the handlebody skein modules are cyclic. In particular, they are generated by invariants. The claim then follows from \cref{prop:reltensorproductcategory}.
	\end{proof}
	
	\subsection{Skein category of the sphere}
	
	In this section we compute the skein category of $S^2$.
	
	\begin{proposition}
		The free cocompletion of the skein category $\SkCat_\cA(S^2)$ is equivalent to the M\"uger center $\ZMug(\fr{\cA})$.
		\label{prop:MugerS2}
	\end{proposition}
	\begin{proof}
		Choose a disk embedding $\DDc\hookrightarrow S^2$ and let $\DDout = S^2 - \DDc$. Then by \cref{prop:closedskeincategory} we may identify
		\[\rZ_\cA(S^2)\cong \lmod{\SkAlgint_\cA(\DDout)}{\fr{\cA}}^{\str}.\]
		The internal skein algebra $\SkAlgint_\cA(\DDout)$ is obtained by monadic reconstruction from the forgetful functor $\SkCat_\cA(\DDout)\rightarrow \cA=\SkCat_\cA(\DDout)$ which is the identity. Therefore, $\SkAlgint_\cA(\DDout)\cong \bu$. The quantum moment map $\mu\colon \cF\rightarrow \bu$ is the map $\SkAlgint_\cA(\Ann)\rightarrow \SkAlgint_\cA(\DDout)$ obtained by embedding $\Ann\hookrightarrow \DDout$. This embedding sends the skein $s_{V, W, f}$ (see \cref{fig:REA}) to a simple skein connecting $V$ and $W$ via $f$. Thus, the moment map in this case is simply the counit $\epsilon\colon \cF\rightarrow \bu$.
		
		Thus,
		\[\rZ_\cA(S^2)\cong \lmod{\bu}{\fr{\cA}}^{\str}.\]
		An object $M\in\fr{\cA}$ is a strongly equivariant $\bu$-module iff $\triv_r(M)$ has the trivial left $\cF$-module structure. By \cref{prop:trivialOqGbimodule} it is equivalent to the condition that $M$ lies in the M\"uger center of $\fr{\cA}$.
	\end{proof}
	
	Recall from \cref{def:HH} the notion of the zeroth Hochschild homology of a category.
	
	\begin{lemma}
		Let $\Sigma$ be a closed oriented surface. Then
		\[\Sk_\cA(\Sigma\times S^1)\cong \HH_0(\SkCat_\cA(\Sigma)).\]
		\label{lm:HHS1}
	\end{lemma}
	\begin{proof}
		Considering the cylinder $\Sigma\times [0, 1]$ as a bordism $\Sigma\coprod \rev{\Sigma}\rightarrow \varnothing$, the relative skein module provides an evaluation pairing
		\[\ev\colon \SkCat_\cA(\Sigma)\otimes \SkCat_\cA(\rev{\Sigma})\longrightarrow \Vect.\]
		Similarly, considering the same cylinder as a bordism $\varnothing\rightarrow \rev{\Sigma}\coprod \Sigma$ we obtain a coevaluation pairing
		\[\coev\colon \SkCat_\cA(\rev{\Sigma})\otimes \SkCat_\cA(\Sigma)\longrightarrow \Vect.\]
		
		Thus, $\Sk_\cA(\Sigma\times S^1)$ is given by the categorical dimension of $\SkCat_\cA(\Sigma)$ which by \cref{rmk:HHdimension} coincides with the zeroth Hochschild homology.
	\end{proof}
	
	Let us now present some corollaries of the computation of the skein category of the sphere.
	
	\begin{corollary}
		For $q$ not a root of unity the $G$-skein module $\Sk_G(S^2\times S^1)$ is one-dimensional.
		\label{cor:skeinS2}
	\end{corollary}
	\begin{proof}
		For $q$ not a root of unity the M\"uger center of $\Repq(G)$ is trivial (see \cref{prop:genericMugercenter}), i.e.
		\[\ZMug(\Repq(G))\cong \Vect.\]
		Therefore, by \cref{prop:MugerS2} we get $\SkCat_G(S^2)\cong \Vect$ for $q$ not a root of unity. Thus, by \cref{lm:HHS1} $\Sk_G(S^2\times S^1)$ is one-dimensional.
	\end{proof}
	
	\begin{corollary}
		Let $N_1$ and $N_2$ be 3-manifolds. For $q$ not a root of unity we have
		\[\Sk_G(N_2\sharp N_1)\cong \Sk_G(N_2)\otimes \Sk_G(N_1).\]
		\label{cor:connectedsum}
	\end{corollary}
	\begin{proof}
		Let $B^3$ be the three-ball and denote $N_1' = N_1\setminus B^3$ and $N_2'=N_2\setminus B^3$. By the TFT property (\cref{thm:WalkerTFT}) we have
		\[\Sk_G(N_2\sharp N_1)\cong \SkMod_G(N_2')\otimes_{\SkCat_G(S^2)} \SkMod_G(N_1').\]
		
		By \cref{prop:MugerS2} and \cref{prop:genericMugercenter} $\fr{\SkCat_G(S^2)}\cong \Vect$. In particular, any $\SkCat_G(S^2)$-module is generated by invariants. Thus, by \cref{prop:reltensorproductcategory} we get
		\[\Sk_G(N_2\sharp N_1)\cong \Sk_G(N_2')\otimes \Sk_G(N_1').\]
		
		The skein module $\Sk_G(B^3)$ is isomorphic to the skein algebra $\SkAlg_G(\DD)$, which is one-dimensional. Therefore, applying the above formula for $N_2 = S^3$ we get
		\[\Sk_G(N_1)\cong \Sk_G(N_1').\]
		Thus,
		\[\Sk_G(N_2\sharp N_1)\cong \Sk_G(N_2)\otimes \Sk_G(N_1)\]
		as required.
	\end{proof}
	
	\subsection{Finite-dimensionality}
	\label{sec:finite-dimensionality}
	The goal of this section is to prove that the skein module of closed oriented 3-manifold is finite-dimensional for generic values of the quantization parameter.  Recall that $\Repqfd(G)$ as a ribbon category is defined over the ring $k=\Z[q^{1/d},q^{-1/d}])$ for some integer $d$.
	
	\begin{theorem}
		Let $M$ be a closed oriented 3-manifold. The $G$-skein module $\Sk_G(M)$ is a finite-dimensional $\Q(q^{1/d})$-vector space.
		\label{thm:witten}
	\end{theorem}
	\begin{proof}
		The dimension of the $\Q(q^{1/d})$-vector space $\Sk_G(M)\rt{\Z[q^{1/d}, q^{-1/d}]}\Q(q^{1/d})$ coincides with the dimension of the $\C\llpar\hbar\rrpar$-vector space $\Sk_G(M)\rt{\Z[q^{1/d}, q^{-1/d}]} \C\llpar\hbar\rrpar$, where $q = \exp(\hbar)$. Denote by $\Repfd_\hbar(G)$ the category of representations of the quantum group over $\k=\C\llbracket\hbar\rrbracket$, where each representation is a free $\k$-module of finite rank. From now on we will drop the subscript $\Repfd_\hbar(G)$ from our notations for skein modules and skein categories.
		
		Choose a Heegaard splitting of $M$. Then we get a closed oriented surface $\Sigma$ of genus $g$, a handlebody $H$ such that $\partial H\cong \Sigma$ and an orientation-preserving diffeomorphism $\sigma\colon \Sigma\rightarrow \Sigma$, so that
		\[M\cong \rev{H}\coprod_\Sigma H.\]
		
		Choose a disk embedding $\DDc\hookrightarrow \Sigma$ and let $\Sigma^*=\Sigma\setminus\DDc$. Without loss of generality we may assume that $\sigma$ restricts to an orientation-preserving diffeomorphism of $\Sigma^*$.
		
		Let $\Skint(H)$ be the internal skein module of $H$, which is a strongly equivariant left $\SkAlgint(\Sigma^*)$-module in $\fr{\cA}$.
		
		The diffeomorphism $\sigma\colon \Sigma^*\rightarrow \Sigma^*$ defines an automorphism of $\SkAlgint(\Sigma^*)$ (denoted by the same letter). Let $\Skint(\rev{H})$ be the internal skein module of $\rev{H}$, which is a strongly equivariant \emph{right} $\SkAlgint(\Sigma^*)$-module in $\fr{\cA}$. By \cref{cor:relativetensorproductMuger} we obtain an isomorphism
		\[\Sk(M)[\hbar^{-1}]\cong \sigma(\Skint(\rev{H}))\rt{\SkAlgint(\Sigma^*)} \Skint(H)[\hbar^{-1}].\]
		
		We will now apply the results of \cref{sect:analysis}. As a Poisson scheme we take $X = G^{2g}$ with the Fock--Rosly Poisson structure. Note that by \cite[Theorem 2.14, Proposition 4.3]{GaJS} the open symplectic leaf of $X$ is given by $\mu^{-1}(G^*)$, where the moment map $\mu\colon G^{2g}\rightarrow G$ is given by
		\[\mu(x_1,y_1,\dots, x_g,y_g) = \prod_i[x_i, y_i]\]
		and $G^*\subset G$ is the big Bruhat cell.
		
		By \cref{prop:fockrosly} $\SkAlgint(\Sigma^*)$ is a flat deformation quantization of $\O(X)$. As an object of $\Rep_\hbar(G)$, we may identify
		\[\SkAlgint(\Sigma^*)\cong \O_\hbar(G)^{\otimes 2g}.\]
		Since $\Rep_\hbar(G)$ is semisimple, we may identify
		\[\O_\hbar(G)\cong \bigoplus_V V^*\otimes V,\]
		where $V$ ranges over isomorphism classes of simple objects of $\Rep_\hbar(G)$. Since each $V$ is free of finite rank as a $k$-module, we conclude that $\O_\hbar(G)$ is $\hbar$-complete and has no $\hbar$-torsion. In a similar way, $\SkAlgint(\Sigma^*)$ is $\hbar$-complete and has no $\hbar$-torsion.
		
		By \cref{thm:handlebody} $\Skint(H)\cong (\O_\hbar(G))^{\otimes g}$. In particular, it is $\hbar$-complete and without $\hbar$-torsion. Moreover, it is a deformation quantization of $L_1 = G^g\subset G^{2g}$. The image of $L_1$ under the moment map is $1\in G$, so $L_1$ is contained in the open symplectic leaf of $G^{2g}$. As $L_1$ is coisotropic and half-dimensional, it is Lagrangian. In a similar way, $L_2 = \sigma(G^g)$ is also Lagrangian. We conclude that
		\[\sigma(\Skint(\rev{H}))\rt{\SkAlgint(\Sigma^*)} \Skint(H) [\hbar^{-1}]\]
		is a finite-dimensional $\C\llpar\hbar\rrpar$-vector space using \cref{thm:holonomictensorproduct}.
	\end{proof}
	
	\begin{corollary}
		The Kauffman bracket skein module $\Sk(M)$ is a finite-dimensional $\Q(A)$-vector space.
	\end{corollary}
	\begin{proof}
		By \cref{prop:KauffmanTL} we may identify $\Sk(M)\cong \Sk_{\SL_2}(M)$ as $\Q(A)$-modules, where $q=A^2$. The claim then follows from \cref{thm:witten}.
	\end{proof}
	
	\section{Discussion}
	\label{sect:discussion}
	
	In this section we collect some remarks about how our results fit in the context of topological field theory, character theory and instanton Floer homology for complex groups. We then discuss an approach for the computation of skein modules using computer algebra.
	
	\subsection{Topological field theory}
	\label{sect:TFT}
	
	In this paper we have used Walker's skein 3-2 TFT for $\cA$ a ribbon category to decompose $\cA$-skein modules on 3-manifolds in terms of a Heegaard splitting. Let us mention some related topological field theories.
	
	\begin{enumerate}[wide, labelindent=0pt]
		\item Walker's skein TFT for an arbitrary ribbon category is not defined on general 4-manifolds (however, the main result of this paper, \cref{thm:witten}, is that it is defined on 4-manifolds of the form $S^1\times M^3$). If we take $\cA$ to be a \emph{modular tensor category}, the theory becomes the Crane--Yetter--Kauffman TFT \cite{CYK}. In fact, in the modular case the TFT is invertible, i.e. it assigns nonzero numbers to closed 4-manifolds, lines to closed 3-manifolds and so on. For example, one may take the modular tensor category associated to the quantum group $\Uq\g$ at a root of unity. In that setting, the 4-dimensional Crane--Yetter--Kauffman TFT carries a boundary theory given by the 3-dimensional Witten--Reshetikhin--Turaev TFT, a mathematical incarnation of Chern--Simons theory for the compact form of $G$. It seems natural to view Walker's TFT associated to the ribbon category $\Repq(G)$ with $q$ generic in the context of analytically continued Chern-Simons theory as discussed in \cite{WittenAnalytic}.
		
		\item The work \cite{BJS} constructs a 3-2-1-0 TFT for an arbitrary rigid braided tensor category. It is conjectured there that for semi-simple ribbon categories their construction coincides with Walker's 3-2 TFT. This conjecture is now proved at the level of surfaces in \cite{Cooke}. However it still remains to compare the functors defined in \cite{BJS} via invocation of the cobordism hypothesis with the concrete formulas from relative skein modules, and it also remains to exhibit Walker's skein category approach as defining a fully local 3-2-1-0 TFT (in which case one might hope to invoke the uniqueness statement in the cobordism hypothesis). We expect that the techniques of blob homology \cite{MorrisonWalker} and the $\beta$ version of factorization homology \cite{AyalaFrancisRozenblyum} might be useful to construct such an extension. We regard these as interesting directions of future inquiry. Note that an arbitrary ribbon category is \emph{not} 4-dualizable, so it does not define a fully extended 4-dimensional TFT. To see this, consider the case $\cA=\Repqfd(\SL_2)$. Then $\SkCat_{\SL_2}(T^2)$ is not 2-dualizable since $\SkAlg(T^2)\cong \Hom_{\SkCat_{\SL_2}(T^2)}(\bu, \bu)$ is infinite-dimensional.
		
		\item One may also consider the derived version of the TFT defined in \cite{BJS} which to a point assigns a version of the derived category of representations of the quantum group. We believe that it is still 3-dualizable, so it should assign complexes to closed 3-manifolds which one may view as ``derived skein modules''. However, we expect that for generic $q$ the derived skein modules are unbounded complexes (i.e. infinite-dimensional), as opposed to the non-derived version.
		
		\item \label{characterTFT} Compactifying the 3-2-1-0 TFT for $\Repq(G)$ on the circle, we obtain a 2-1-0 TFT which assigns to the point $\HC_q(G)$, the monoidal category of $q$-Harish-Chandra bimodules (we refer to this as the \emph{$q$-$G$-character theory}). This theory has a degeneration (the \emph{$G$-character theory}) where we replace $\HC_q(G)$ by $\HC(G)$, the monoidal category of Harish-Chandra bimodules. The derived version of this TFT was studied in \cite{BZN, BZGN}. See \cref{sec:charactertheory} below for further details.
		
		\item Kapustin and Witten \cite{KapustinWitten} have studied a topological twist (first described by Marcus in \cite{Marcus}) which is parametrized by a number $t\in\CP^1$ of the 4d $\cN=4$ supersymmetric Yang--Mills theory (for a compact form of a complex simple simply-connected group $G$) with complexified coupling constant $\tau$. They have shown that the corresponding topological field theory only depends on a combination of $t$ and $\tau$
		\[\Psi = \frac{\tau+\overline{\tau}}{2} + \frac{\tau-\overline{\tau}}{2}\left(\frac{t-t^{-1}}{t+t^{-1}}\right)\in\CP^1\]
		and that the $S$-duality in the Yang--Mills theory after the twist is related to the geometric Langlands duality. We refer to \cite{ElliottYoo} for a study of the spaces of classical solutions in this topological field theory from the perspective of derived algebraic geometry. We expect that for generic $\Psi$ the space of states on a 3-manifold $M$ in this TFT is related to the derived $G$-skein module for $q=\exp\left(\frac{\pi i}{\Psi r_G}\right)$, where $r_G$ is the lacing number (the ratio of the norm squared of the long root to that of the short root).
		
		\item Vafa and Witten \cite{VafaWitten} have studied another topological twist (first described by Yamron in \cite{Yamron}) of the $\cN=4$ supersymmetric 4d Yang--Mills theory. The restriction of the Vafa--Witten TFT and the $t=0$ Kapustin--Witten TFT to 3-manifolds coincide (see \cite[Section 5.3]{Setter} and \cite[Example 4.32]{ElliottSafronov}), so the previous remark applies to the Vafa--Witten theory as well.
	\end{enumerate}
	
	\subsection{Character theory}\label{sec:charactertheory}
	It was shown in \cite{BZGN} that the assignment of the derived category of $G$-Harish-Chandra bimodules to a point defines an oriented 2-1-0 TFT which to a closed 2-manifold $\Sigma$ assigns the Borel--Moore homology of the $G$-character stack. By the remarks in \cref{characterTFT} above one may thus consider the skein module $\Sk_G(\Sigma \times S^1)$ as a $q$-deformation of the zeroth Borel-Moore homology of $\Loc_G(\Sigma)$. It is natural to apply the same methods to study both invariants. 
	
	An interesting feature of the $G$-character theory (respectively $q$-$G$-character theory) is that it admits an additional continuous parameter, arising from the spectrum of the commutative algebra $\U(\g)^G$ (respectively $\O_q(G)^G$) acting on Harish-Chandra bimodules. For example, fixing the generalized eigenvalues in the character theory to $0 \in \h/W = \Spec(\U(\g)^G)$ gives a theory (the \emph{unipotent} character theory) which assigns the finite Hecke category to a point and the category of Lusztig's unipotent character sheaves to a circle. Studying the $q$-analogue of these objects is an interesting area for further study (see \cite{GJV} for a discussion of $q$-character sheaves). 
	
	The theories obtained by fixing the eigenvalues in the character theory appear to enjoy an extra degree of finiteness. One indicator of this is given by the truncated 2-sided cell fusion categories of \cite{LusztigCells,BFO} (note that fusion categories define 3-2-1-0 TFTs \cite{DSPS}). Understanding the character theory as a family over this space of parameters is the subject of an ongoing project of the first author with David Ben-Zvi. An appropriate $q$-analogue of these ideas suggests that one can compute the dimensions of the skein modules for manifolds of the form $\Sigma \times S^1$ using the $3$-manifold invariants associated to certain fusion categories. It would be interesting to compare these predictions with the lower bounds for the dimensions of such skein modules given by Gilmer and Masbaum \cite{GilmerMasbaum}.
	
	\subsection{Complexified instanton Floer homology}\label{sec:complexified-instanton-floer-homology}
	
	Let $\Sigma$ be a closed oriented surface. Then the skein algebra $\SkAlg(\Sigma)$ is a deformation quantization of the Goldman (equivalently, Atiyah--Bott or Fock--Rosly) Poisson structure on the character variety $\underline{\Loc}_{\SL_2}(\Sigma)$ \cite{TuraevSkein,BFKB}, so that $\SkAlg(\Sigma)|_{A=-1}\cong \O(\underline{\Loc}_{\SL_2}(\Sigma))$. In a similar way, we may view the skein category $\SkCat(\Sigma)$ as a deformation quantization of the $0$-shifted symplectic structure \cite{PTVV} on the character \emph{stack} $\Loc_{\SL_2}(\Sigma)$, so that $\fr{\SkCat(\Sigma)}|_{A=-1}\cong \QCoh(\Loc_{\SL_2}(\Sigma))$.
	
	Now consider a Heegaard splitting $M = N_2\cup_{\Sigma} N_1$ of a closed oriented 3-manifold. The character stack $\Loc_{\SL_2}(M)$ in this case has a $(-1)$-shifted symplectic structure and the restriction maps
	\[\Loc_{\SL_2}(N_1), \Loc_{\SL_2}(N_2)\longrightarrow \Loc_G(\Sigma)\]
	are $0$-shifted Lagrangian, so that we have a derived Lagrangian intersection
	\[\Loc_{\SL_2}(M)\cong \Loc_{\SL_2}(N_2)\times_{\Loc_{\SL_2}(\Sigma)} \Loc_{\SL_2}(N_1).\]
	
	It was shown by Bullock \cite{BullockRings} and Przytycki and Sikora \cite{PrzytyckiSikora} that the $A=-1$ specialization of the skein module $\Sk(M)$ is isomorphic to the algebra of functions $\O(\Loc_{\SL_2}(M))$ on the character variety (equivalently, character stack). Passing to the derived level, we may view the \emph{derived} skein module as a BV quantization \cite[Section 7]{CostelloGwilliam} of the $(-1)$-shifted symplectic structure on the character stack $\Loc_{\SL_2}(M)$.
	
	One model of such a BV quantization is constructed by Ben-Bassat, Brav, Bussi and Joyce \cite{BBBBJ} given the choice of \emph{orientation} of $\Loc_{\SL_2}(M)$ (which is automatic in our context). Namely, using their results one may construct a perverse sheaf $\cP^\bullet(M)$ on the classical stack $t_0(\Loc_{\SL_2}(M))$. We expect that the hypercohomology of this perverse sheaf is closely related to derived skein modules (in fact, we expect such a relationship to hold for any $G$).
	
	A version of this approach was realized by Abouzaid and Manolescu \cite{AbouzaidManolescu}. Namely, they consider a subset $\underline{\Loc}_{\SL_2}^{irr}(\Sigma)\subset \underline{\Loc}_{\SL_2}(\Sigma)$ of irreducible local systems, which is a complex symplectic manifold. The image of the character varieties of handlebodies $\underline{\Loc}_{\SL_2}(L_1)$ and $\underline{\Loc}_{\SL_2}(L_2)$ in $\underline{\Loc}_{\SL_2}(\Sigma)$ then define Lagrangian subvarieties $L_1, L_2\subset \underline{\Loc}_{\SL_2}^{irr}(\Sigma)$, so that
	\[\underline{\Loc}_{\SL_2}^{irr}(M)\cong L_2\cap L_1.\]
	
	Given two Lagrangians $L_1, L_2$ in a complex symplectic manifold $X$, Bussi \cite{Bussi} has constructed a perverse sheaf on $t_0(L_2\times_X L_1) = L_2\cap L_1$ which is equivalent to the perverse sheaf of \cite{BBBBJ} on the derived Lagrangian intersection $L_2\times_X L_1$. Using these results Abouzaid and Manolescu have constructed a perverse sheaf $P^\bullet(M)$ on $\underline{\Loc}_{\SL_2}^{irr}(M)$ which they have shown is independent of the Heegaard splitting of $M$. The relationship between the hypercohomologies of $P^\bullet(M)$ and $\cP^\bullet(M)$ may thus be viewed as an $\SL(2, \C)$-version of the Atiyah--Floer conjecture. 
	%added below 8/18/2020
	Abouzaid and Manolescu also define a framed version $HP^\bullet_\sharp(M)$ of their construction, corresponding to the derived intersection of Lagrangians inside the representation variety of the punctured surface. In forthcoming work of the first and third named authors, we will establish an equivalence between $HP_\sharp(M)$ and the derived tensor product of internal skein modules as explained in \cref{rem:derived}.
	
	\subsection{Computer algebra}\label{sec:comp-alg} Computers perform remarkably well as algebraists.  As topologists, less so.  The essential ``3-dimensionality" in the definition of skein modules makes it very difficult to use computer algebra to study them: computer algebra packages are well equipped to work in \emph{one dimension} -- that is, to computing with non-commutative associative algebras and their modules, bimodules, etc, but how does one program into a computer a vector space spanned by links in a 3-manifold?
	
	The relative tensor product formula of \cref{cor:tensorproduct} provides a relatively straightforward and elementary algebraic ``one-dimensional" algorithm for computing skein modules, as well as a theoretical proof that said algorithm terminates.  To illustrate this and in order to generate new conjectures about skein modules and their dimensions, we have written a program in MAGMA to implement this algorithm. We have uploaded the source code here:  \url{http://www.maths.ed.ac.uk/~djordan/skeins}.
	
\providecommand{\bysame}{\leavevmode\hbox to3em{\hrulefill}\thinspace}
\providecommand{\MR}{\relax\ifhmode\unskip\space\fi MR }
\providecommand{\MRhref}[2]{%
	\href{http://www.ams.org/mathscinet-getitem?mr=#1}{#2}
}
\providecommand{\href}[2]{#2}

\end{document}